\documentclass{article}
\pdfoutput=1

\usepackage{a4wide, amsthm, color, verbatim, multicol, hyperref, url}
\usepackage{difftrees}

\usepackage[small]{diagrams}

\newtheorem{theorem}{\bf{Theorem}}[section]
\usepackage{hanshopf}

\newtheorem{lemma}[theorem]{\bf{Lemma}}
\newtheorem{proposition}[theorem]{\bf{Proposition}}

\newtheorem{algorithm}[theorem]{\bf{Algorithm}}

\theoremstyle{definition}
\newtheorem{example}[theorem]{\bf{Example}}

\newtheoremstyle{remark}{\topsep}{\topsep}
{}
{}
{\bfseries}
{.}
{ }
{\thmname{#1}\thmnumber{ #2}\thmnote{ #3}}

\theoremstyle{remark}
\newtheorem{remark}[theorem]{\bf{Remark}}

\theoremstyle{definition}
\newtheorem{definition}[theorem]{\bf{Definition}}

\numberwithin{equation}{section}

\def \B+{\operatorname{B}}
\def \Hom{\operatorname{Hom}}
\def \End{\operatorname{End}}

\def \Id{\operatorname{Id}}

\begin{document}

\begin{titlepage}

\thispagestyle{empty}

\title{Hopf algebras of formal diffeomorphisms and numerical integration on manifolds}

\author{Alexander Lundervold and Hans Munthe-Kaas\\ {\small Department of Mathematics, University of Bergen,}\\ {\small Johannes Brunsgate 12, N-5008 Bergen, Norway} \\ {\small \{alexander.lundervold,hans.munthe-kaas\}@math.uib.no}}

\date{}

\maketitle

\begin{abstract}
B-series originated from the work of John Butcher in the 1960s as a tool to analyze numerical integration of differential equations, in particular Runge--Kutta methods. Connections to renormalization theory in perturbative quantum field theory have been established in recent years. The algebraic structure of classical Runge--Kutta methods is described by the Connes--Kreimer Hopf algebra.

Lie--Butcher series are generalizations of B-series that are aimed at studying Lie-group integrators for differential equations evolving on manifolds. Lie group integrators are based on general Lie group actions on a manifold, and classical Runge--Kutta integrators appear in this setting as the special case of $\RR^n$ acting upon itself by translations. Lie--Butcher theory combines classical B-series on $\RR^n$ with Lie-series on manifolds. The underlying Hopf algebra $H_N$ combines the Connes--Kreimer Hopf algebra with the shuffle Hopf algebra of free Lie algebras.

Aimed at a general mathematical audience, we give an introduction to Hopf algebraic structures and their relationship to structures appearing in numerical analysis. In particular, we explore the close connection between Lie series, time-dependent Lie series and Lie--Butcher series for diffeomorphisms on manifolds. The role of the Euler and Dynkin idempotents in numerical analysis is discussed. A non-commutative version of a Fa\`a di Bruno bialgebra is introduced, and the relation to non-commutative Bell polynomials is explored.

\end{abstract}

\end{titlepage}


\tableofcontents

\section{Outline}
The main point of this paper is to explore algebraic structures underlying groups of formal diffeomorphisms on manifolds. The focus is on some important mathematical structures appearing in numerical integration on manifolds that are likely to find applications also in other areas of mathematics. The relationship between classical Lie series on manifolds, time-dependent Lie series and Lie--Butcher series is explained in detail. We develop the algebraic structures introduced in \cite{munthe-kaas95lbt,munthe-kaas98rkm,munthe-kaas03oep,munthe-kaas08oth,owren99rkm,owren05ocf,berland05aso}, and in particular explore connections between Hopf- and Lie algebras, differential geometry and analysis of numerical integration on manifolds. The paper does not go into a {detailed} study of applications of these algebraic structures in numerical analysis, but we sketch briefly several of the many applications in this field.

The introductory Chapters 2 and 3 give an overview of well-known results. Chapter 2 contains a brief introduction to numerical integration and algebraic structures appearing in  numerical analysis, both classical methods on $\RR^n$ and Lie group methods generalizing to manifolds. Chapter 3 presents a brief introduction to Hopf algebraic structures. 

Chapter 4 contains more new and recent material. It details the algebraic structures of Lie--Butcher theory, and discusses the interplay between algebraic and differential geometric points of view. In particular, we want to emphasize the strong connections between the algebraic theory of Lie series, time-dependent Lie series and Lie--Butcher series. Chapter 4 therefore starts with a discussion of classical Lie series and pullback formulas on manifolds, continuing with an exploration of some less known time-dependent pullback formulas. We will explain the relevance of the Euler and Dynkin idempotents in numerics and introduce a non-commutative \emph{Dynkin--Fa\`a di Bruno} bialgebra, related to non-commutative Bell polynomials appearing in various contexts in earlier works: numerical analysis~\cite{munthe-kaas98rkm}, control theory~\cite{monaco07fcc} and quantization~{\cite{lenczewski1998anl}}. This Dynkin--Fa\`a di Bruno bialgebra is related to, but different from, the Hopf algebras explored by Brouder et.\ al.\ in~\cite{brouder06nch}. 

In the final part of Chapter 4, we turn to Lie--Butcher series. We explore backward error analysis and the substitution law in the setting of algebras of non-commuting frames on manifolds. Although we will not give detailed expositions and applications of these subjects, we hope that this presentation will systematize the theory and open the topics for further research.

\section{Introduction to numerical integrators and their analysis}\label{sec:introNA}

Let $\M$ be a manifold and $F\colon \M\rightarrow \TM$ a vector field. By the flow of an autonomous vector field $F$ we mean the diffeomorphism $\Phi_{t,F}\colon \M\rightarrow\M$, defined for $t\in\RR$ such that $\Phi_{s,F}\opr \Phi_{t,F} = \Phi_{s+t,F}$, $\Phi_{0,F} = \id$ and 
$\left.\partial/\partial t\right|_{t=0}\Phi_{t,F}(p) = F(p)$ for all $p\in \M$.

Numerical integration of ODEs is about constructing good numerical approximations to $\Phi_{t,F}$ for a given vector field $F$. A numerical integration algorithm yields a diffeomorphism $\Psi_{h,F}$, henceforth called the \emph{numerical integrator}. The real parameter $h$ is called the \emph{step size}. For an initial point $y_0\in \M$, and a chosen step size $h>0$, the numerical method produces a discrete sequence of solution points $y_i = \Psi_{h,F}(y_{i-1})$, with the goal of arriving at $y_k \approx \Phi_{kh,F}(y_0)$. Note that, unlike the exact flow, numerical integrators are \emph{not} 1-parameter Lie groups in $h$. In general we have $\Psi_{h,F}\opr \Psi_{s,F} \neq \Psi_{h+s,F}$, and $\Psi_{-h,F}\neq\Psi_{h,F}^{-1}$. Integrators for which the latter identity holds are called  \emph{(time-)symmetric} methods. Most integrators satisfy the consistency conditions $\Psi_{0,F} = \id$ and $\left.\partial/\partial t\right|_{t=0}\Psi_{t,F}(p) = F(p)$ as well as scaling homogeneity $\Psi_{h,F}=\Psi_{1,hF}$.

Many algebraic aspects of numerical integration are related to the computation of compositions, logarithms and exponentials of numerical integrators. In this introduction we will introduce some basic algebraic structures arising in the analysis of numerical integrators. In particular we will focus on structures that originate from the study of \emph{Lie group integrators}, which are numerical integrators on general manifolds. The resulting theory combines Lie theory with the classical Butcher theory that describes numerical integrators on $\RR^n$. 

This first section presents a survey of well known results from numerical analysis. A detailed understanding of this introductory section is not necessary for reading the rest of the paper, and readers mainly interested in algebraic structures may jump directly to Section~\ref{sec:hopf}.

\subsection{Classical integrators}
In the early 1960s, John Butcher set out to explore the algebraic territory of numerical algorithms for integrating ODEs evolving on vector spaces
\begin{equation}\label{ode}
y'(t) = F(y),\qquad  y\in\RR^n,\quad F\colon \RR^n\rightarrow\RR^n.
\end{equation}
In particular he studied the family of Runge--Kutta methods.  Given a time step $h\in\RR$, these methods advance the solution from $y_0 = y(0)$ to $y_1\approx y(h)$ as: 
\begin{proglist}
{for } $r  =  1 : s$ {do}\+\\
$Y_r  = \sum_{k=1}^s a_{rk}F_k + y_0$\\
   $F_r   =  h F(Y_r)$\-\\
{end}\\
$y_1  =  \sum_{k=1}^s b_{k}F_k + y_0$.
\end{proglist}
This basic step is iterated: $y_0\mapsto y_1\mapsto\ldots\mapsto y_n$, with  constant or variable step sizes $h$, until the final solution $y_n \approx y(t_n)$ is reached. The coefficients $a_{rk}$ and $b_k$ for $r,k\in\{1,\ldots,s\}$ define a particular $s$-stage RK method.

A goal of numerical analysis is to characterize coefficients $a_{rk}$ and $b_k$ that yield `good' methods (when applied to a given class of differential equations). The view of what a good integration method is has, however, evolved over the last decades. 
Traditionally, \emph{order theory} and \emph{stability} were the most important properties to consider.  A numerical integrator is of order $p$ if the first $p+1$ terms of the Taylor expansion of the analytical solution agrees with the first $p+1$ terms of the numerical method (developed in the parameter $h$). Requiring a certain order results in algebraic conditions, called \emph{order conditions}, on the coefficients of the method.

Many numerical methods for solving the equation (\ref{ode}) can be studied by using \emph{B-series} (see e.g. \cite{hairer06gni}), introduced by Hairer and Wanner in 1974~\cite{hairer74otb}. A B-series is a (formal) series indexed over the set of rooted trees $T$, and can for a vector field $F$ be written as 
\begin{equation}\label{eq:bserclassic}B_{h}(a)(y) = a(\one)y + \sum_{\tau \in T} \frac{h^{|\tau|}}{\sigma(\tau)} a(\tau) \F_F(\tau)(y).\end{equation}
Here $a$ is a map $a: T \rightarrow \mathbb{R}$, $\one$ is the empty tree, $|\tau|$ is the number of vertices of $\tau$ (the \emph{order} of the tree $\tau$) and $\sigma$ is a certain symmetry factor. The map $\F_F(\tau): \mathbb{R}^n \rightarrow \mathbb{R}^n$ is the \emph{elementary differential} of the tree $\tau$ and is given recursively as follows:
\begin{equation}\label{eg:elmdiffB}\F_F(\tau) = F^{(m)}\left(\F_F(\tau_1)(y), ..., \F_F(\tau_m)(y)\right)(y),\end{equation} where $\tau = B^+(\tau_1,...,\tau_m)$ is the tree constructed by adding a common root to the subtrees $\tau_1\ldots\tau_m$, and $F^{(m)}$ is the $m$th derivative of the vector field.

One way in which B-series can be applied to the study of numerical methods is to order theory. For example, the order conditions for Runge--Kutta methods can easily be obtained by writing the method as a B-series and then comparing the coefficients of this series with the exact solution written as a B-series (see e.g. \cite[Chapter. III.1.2]{hairer06gni}).

The composition of Runge--Kutta methods is also of great interest, and this leads to the study of the composition of B-series. A series $B_{hF}(a)$ is inserted into another series $B_{hF}(b)$, which gives the B-series $B_{hF}(a)(B_{hF}(y)(b)) = B_{hF}(a\cdot b)(y)$. The resulting product $a\cdot b$ gives rise to a group, called the \emph{Butcher group} \cite{butcher72aat, chartier08aat}.

Butcher realized early on that the set of Runge--Kutta methods forms a group, and characterized algebraically the composition and inverse in this group. Much later, this group was identified with the character group of the Connes--Kreimer Hopf algebra \cite{connes1998har, dur86mfi,brouder04tra}.

In recent years the importance of preserving various geometric properties of the underlying continuous dynamical system has become better understood. The research topic \emph{Geometric Numerical Integration}~\cite{hairer06gni} emphasizes this view. Geometric integration algorithms have been successfully developed for various classes of differential equations, such as volume preserving flows, Hamiltonian equations, systems with first integrals and equations evolving on manifolds.
An important tool for investigating the geometrical properties of a numerical integrator is through \emph{backward error analysis}. For a given numerical method $\Psi_{h,F}$, we seek a series expansion of a modified vector field $(h,F)\mapsto \tilde{F}_h$ such that the numerical solution equals\footnote{The series for $\tilde{F}_h$ is a formal series which usually does not converge. By truncating the series at an optimal point we find a modified equation which is exponentially close to the numerical solution, see~\cite{hairer06gni}. In this paper we deal only with formal series, and convergence is not considered.} the analytical flow of the modified vector field: $\Psi_{h,F} = \left.\Phi_{t,\tilde{F}_h}\right|_{t=h}$. 

This is computed as a formal logarithm $\tilde{F}_h = \Log(\Psi_{h,F})$, which in  Hopf algebraic language is expressed by the \emph{Eulerian idempotent} (Section~\ref{eulidem}).

Still another idea, which has been developed in~\cite{chartier2005asl}, is to ask for a series development of a modified vector field $\overline{F}_h$ such that when the numerical method is applied to $\overline{F}_h$, the exact analytical solution is produced: $\Psi_{h,\overline{F}_h}=\Phi_{h,F}$. This has been taken much further in recent work~\cite{chartier07nib,calaque08tha}. The algebraic operation $(h,F)\mapsto \overline{F}_h$  is commonly referred to as a \emph{substitution law}. The Hopf algebra of the substitution law is introduced in~\cite{calaque08tha}.

The theory of B-series is often a very important component in a numerical analyst's toolbox, and is used to study all of the above: order theory, backward error analysis, modified vector fields and structure preserving properties of numerical integrators.

\subsection{Lie group integrators}
Numerical Lie group integrators for ODEs is a generalization of numerical integration of ODEs from the classical setting of equations on $\RR^n$ to differential equations on manifolds. See~\cite{iserles00lgm} for an extensive survey.

\subsubsection{Exponential Euler method}\label{sec:eeuler}
Let $\M$ denote a manifold, $\XM$ its Lie algebra of vector fields with the Jacobi bracket and $\DM$ the group of diffeomorphisms on $\M$. Let $\signature{\exp}{\XM}{\DM}$ denote the flow operator. We want to numerically integrate an ODE on $\M$  given as
\begin{equation}\label{eq:odeonM} 
y'(t) = F(y), \quad y(0) = y_0\quad\mbox{ for $F\in \XM$,}
\end{equation}
with the analytical solution $y(t) = \exp(tF)\dpr y_0$. Here $\exp(tF)\dpr y_0$ denotes the evaluation of the diffeomorphism $\exp(tF)$ at $y_0\in \M$.

\begin{assumption}
The fundamental assumption for numerical Lie group integrators is the existence of a subalgebra $\g \subset \XM$ such that
\begin{itemize}
\item All vector fields $V\in \g$ can be exponentiated exactly.
\item The Lie algebra $\g$ defines a \emph{frame} on $\TM$, i.e.\ $\g$ spans the tangentspace $T_p\M$ at all points $p\in \M$. In other words, the action generated by $\g$ is transitive on $\M$.
\end{itemize}
\end{assumption}

The vector fields in $\g$ are called the \emph{frozen vector fields}. Due to the frame assumption, we can always express the vector field $F$ and the ODE~(\ref{eq:odeonM}) in terms of frozen vector fields via a function $\signature{f}{\M}{\g}$ as $F(y) = f(y)\dpr y$, where $f(y)\in \XM$ and $f(y)\dpr y$ denotes evaluation of this vector field in $y$. Thus~(\ref{eq:odeonM}) can be written in the form
\begin{equation}\label{eq:odeonM2}
y'(t) = f(y)\dpr y, \quad y(0)  = y_0\in \M .
\end{equation}
In the case where $\g$ forms a basis for $T_y\M$, the function $f(y)$ is uniquely defined. In more general situations, $\g$ is an overdetermined frame for $T_y\M$, and there is a freedom in the choice of $f$. This is called a choice of isotropy, and is of major importance for the quality of the numerical integrator. 

With the equation written as~(\ref{eq:odeonM2}), we can present the simplest of all Lie group integrators: the \emph{exponential Euler method}. Given a time step $h\in\RR$, the method advances the solution from $y_0 = y(0)$ to $y_1\approx y(h)$ as:
\begin{algorithm}[Exponential Euler]
\[   y_1 = \exp(h f(y_0))\dpr y_0.\]
\end{algorithm}
In each step the solution is advanced $y_k\mapsto y_{k+1}$ by integrating the frozen vector field equation 
\[ y'(t) = f(y_k)\dpr y, \quad y(0) = y_k\] from $t=0$ to $t=h$. We will in the sequel present methods of higher order and with superior qualities compared to this simple scheme. The main theme of the paper is the algebraic structures arising from the numerical analysis of such integration schemes.

\subsubsection{Choosing a good action}
In practice it is of importance that $\exp \colon \g \rightarrow\DM$ can be computed fast, and furthermore that the given vector field $F(y)$ is locally well approximated by $f(y_0)\dpr y$. Exactly what `well' means depends on what we want to achieve. In many situations we can choose $\g$ so that certain first integrals of the original system are exactly preserved by the frozen flows. Choosing $\g$ and $\signature{f}{\M}{\g}$ is in many ways similar to choosing a \emph{preconditioner} in iterative methods for solving linear equations: we want a good approximation which is easy to compute.

A simple choice of $\g$ is obtained by embedding $\M\subset \RR^N$ and choosing $\g = \RR^N$ as the (commutative) algebra generated by $\{\partial/\partial x_j\}_{j=1}^N$, i.e., the constant vector fields on $\RR^N$. Since the vector fields are constant, we have $f(y_0)\dpr y = f(y_0)$, so the function $f$ simply becomes $f(y_0) = F(y_0)\in \RR^N$, all commutators in $\g$ vanish and the exponential on $\g$ is $\exp(V)\dpr p = V+p$ for $V,p \in \RR^N$. In this case all Lie group integrators will reduce to classical integrators, e.g.\ exponential Euler becomes the classical Euler $y_1 = hF(y_0) + y_0$.

The other extreme is $\g=\XM$ and $f(y)=F$ for all $y$,  in which case exponential Euler yields the analytical solution exactly. However, the  computation of the exponential on $\g$ is just as difficult as solving the original equation.
We seek efficient choices in between these two extremes.

In many cases $\g$ is given as the infinitesimal generators of a (e.g.\ left) Lie group action on $\M$. For example, consider the sphere $\M = S^2$ acted upon from left by the group $G=\operatorname{SO}(3)$ of orthogonal $3\times 3$ matrices, whose Lie algebra ${\mathfrak so}(3)$ consists of skew $3\times 3$ matrices. Any matrix $V\in {\mathfrak so}(3)$ is uniquely identified with the infinitesimal generator\footnote{Recall that the identification of the Lie algebra of a left group action with the infinitesimal generator in $\XM$ is an anti-homomorphism, $[\xi_V,\xi_W] = -\xi_{[V,W]}$. In this paper the brackets are Jacobi brackets on $\XM$, and some signs may differ when compared to cited papers.} $\xi_V\in \XM$, given  by matrix-vector multiplication $\xi_V\dpr y = Vy$ for $y\in S^2$. Therefore~(\ref{eq:odeonM2}) becomes $y'(t) = V(y) y$, where $V(y)$ is a skew symmetric matrix. The exponentiation is related to the matrix exponential $\operatorname{expm}(V)$ as $\exp(\xi_V)\dpr y = \operatorname{expm}(V)y$. 

Another important example of group actions arise in the solution of \emph{isospectral differential equations}, where $\operatorname{GL}(n)$ acts on  ${\mathfrak gl}(n)$ by the \emph{adjoint action} (similarity transform) $A\cdot Y = AYA^{-1}$ for $A\in \operatorname{GL}(n)$, $Y\in {\mathfrak gl}(n)$. In these problems $\M\subset {\mathfrak gl}(n)$ is one of the ({isospectral}) orbits of the action. For this action~(\ref{eq:odeonM2}) acquires the isospectral form  $Y'(t)  = [B(Y),Y]$ for some $B(Y)\in{\mathfrak gl}(n)$. Since the action is a similarity transform it is guaranteed that all the eigenvalues of $Y(t)$ are preserved also by the numerical integrator.

Yet another example, which occurs in \emph{Lie--Poisson problems} in computational mechanics, is the \emph{coadjoint action} of a Lie group on the dual of its Lie algebra, $\g^*$. In this case $\M \subset \g^*$ is a coadjoint orbit. Using this action we can guarantee that the numerical Lie group integrator exactly preserves the Casimirs of the continuous system.

For other problems it may be advantageous to choose $\g$ by simplifying the original equation to a family of integrable equations. An example is the computation of the motion of charged particles in a magnetic field. The solution in the case of constant magnetic fields is given by helical motions around the field lines. The corresponding Lie algebra yields fast and accurate Lie group integrators for the full problem of non-constant magnetic fields.  Another example is integration of a spinning top, where we obtain simpler equations by considering the direction of gravity as being constant in body coordinates. In both these problems, the action preserves important first integrals of the system. A third example is integration of stiff equations on $\RR^n$, where an integrable Lie algebra is obtained by considering all affine linear vector fields. This connects the theory of Lie group integrators with the so-called \emph{exponential integrators}.
See~\cite{iserles00lgm} for details.

\subsubsection{Higher order methods}
Most Lie group methods for integrating~(\ref{eq:odeonM2}) are built from linear operations and commutators in $\g$ and compution of flows of frozen vector fields (exponentials). Runge--Kutta type methods with basic motions expressed in terms of an exponential of a sum of elements in $\g$ are commonly referred to as RKMK methods~\cite{iserles00lgm}, as in  the following example:
\begin{algorithm}[4th order RKMK from~\cite{munthe-kaas98rkm}]\label{alg:rk4}
\[\begin{array}{ll}
Y_1 = y_0 &\ F_1 = hf(Y_1) \\ 
Y_2  = \exp(\frac12 F_1)\dpr y_0 &\ F_2  =  hf(Y_2)\\
Y_3  =  \exp(\frac12 F_2 + \frac{1}{24} [F_1,F_2])\dpr y_0 &\ F_3  =  hf(Y_3)\\
Y_4  =  \exp(F_3 + \frac16 [F_1,F_3])\dpr y_0 &\
F_4 = hf(Y_4)\\
V = \frac16 F_1 + \frac13 (F_2 +  F_3) + \frac16 F_4 &\
I  =  \frac18 F_1 + \frac{1}{12}(F_2 +F_3)-\frac{1}{24} F_4\\
y_1 =  \exp(V+[I,V])\dpr y_0 .
\end{array}
\]
\end{algorithm}
Methods where the basic motions are products of exponentials of simple elements in $\g$ are called \emph{Crouch--Grossman methods}~\cite{crouch93nio,owren99rkm}. 
\begin{algorithm}[3rd order Crouch--Grossman method from~\cite{owren99rkm}]
\[\begin{array}{ll}
Y_1 = y_0 &\ F_1 = hf(Y_1) \\ 
Y_2  = \exp(\frac34 F_1)\dpr y_0 &\ F_2  =  hf(Y_2)\\[.5ex]
Y_3  =  \exp(\frac{119}{216}F_2)\dpr\exp(\frac{17}{108}F_1)\dpr y_0 &\ F_3  =  hf(Y_3)\\[.5ex]
y_1 =  \exp(\frac{13}{51}F_3)\dpr\exp(-\frac{2}{3}F_2)\dpr\exp(\frac{24}{17}F_3) y_0 .
\end{array}
\]
\end{algorithm}
More recently methods have been developed which combine exponentials of sums and products of exponentials, as in the \emph{commutator free Lie group methods}~\cite{celledoni2003cfl}. An example is:
\begin{algorithm}[4th order commutator free method from~\cite{celledoni2003cfl}]
\[\begin{array}{ll}
Y_1 = y_0 &\ F_1 = hf(y_0)\\[.5ex]
Y_2  = \exp(\frac12 F_1)\dpr y_0 &\ F_2  =  hf(Y_2)\\[.5ex]
Y_3  =  \exp(\frac12 F_2)\dpr y_0 &\ F_3  =  hf(Y_3)\\[.5ex]
Y_4  =  \exp(-\frac12 F_1+F_3)\dpr Y_2 &\ F_4  =  hf(Y_4)\\[.5ex]
\multicolumn{2}{c}{y_1  =  \exp(\frac14 F_1+\frac16 (F_2+F_3)-\frac{1}{12}F_4)\dpr
\exp(-\frac{1}{12} F_1+\frac16 (F_2+F_3)+\frac14 F_4)\dpr y_0 .}
\end{array}
\]
\end{algorithm}
For equations of \emph{Lie type}, $y'(t) = f(t)\cdot y$, numerical methods based on Magnus and Fer expansions have been developed in~\cite{iserles1999sld, iserles1999imm}, and the algebraic theory has recently been developed further in~\cite{ebrahimi-fard2009ama}.

To study order conditions, backward error analysis and structure preservation of such methods, it is important to understand 
B-series in a general setting of group actions on manifolds. A first attempt at combining Lie and B-series in a common mathematical framework appeared in~\cite{munthe-kaas95lbt,munthe-kaas98rkm}. Hopf algebraic aspects have been explored further in \cite{munthe-kaas03oep,berland05aso,munthe-kaas08oth}.

\section{Hopf algebras}\label{sec:hopf}
This section gives a short collection of some facts and properties of Hopf algebras that we will use in this work. For a more thorough introduction, see e.g. \cite{abe80ha}, \cite{manchon06haf}, \cite{sweedler69ha}, \cite{kassel95qg}, \cite{cartier2006apo}.

\subsection{Basic definitions}\label{basics}

Let $k$ be a field containing $\mathbb{Q}$. 

\begin{definition}
A \emph{$k$-algebra} $A$ consists of a $k$-vector space $A$ together with two maps $\mu: A \otimes A \rightarrow A$ and $\eta: k \rightarrow A$, called the \emph{product} and the \emph{unit} of $A$, such that:
\begin{itemize}
\item[(i)] $\mu$ is associative, i.e. $\mu \circ (I \otimes \mu) = \mu \circ (\mu \otimes I)$,
\item[(ii)] the composites $A \cong A \otimes k \overset{I \otimes \eta}{\longrightarrow} A \otimes A \overset{\mu}{\rightarrow} A$ and $A \cong A \otimes k \overset{\eta \otimes I}{\longrightarrow} A \otimes A \overset{\mu}{\rightarrow} A$ both equal $I$.
\end{itemize}
Here $I$ denotes the identity map.
\end{definition}

\noindent An algebra $A$ is called commutative if $\mu \circ \tau = \mu$, where $\tau$ is the \emph{flip} map $\tau(a_1 \otimes a_2) = a_2 \otimes a_1$.

\begin{definition}
A \emph{$k$-coalgebra} $C$ is a $k$-vector space equipped with two maps $\Delta: C \rightarrow C \otimes C$ and $\epsilon: C \rightarrow k$, the \emph{coproduct} and the \emph{counit}, such that: 
\begin{itemize}
\item[(i)] $\Delta$ is coassociative, i.e. $(\Delta \otimes I) \circ \Delta = (I \otimes \Delta) \circ \Delta$,
\item[(ii)] the composites $C \overset{\Delta}{\rightarrow} C \otimes C \overset{I \otimes \epsilon}{\longrightarrow} C \otimes k \cong C$ and 
$C \overset{\Delta}{\rightarrow} C \otimes C \overset{\epsilon \otimes I}{\longrightarrow} k \otimes C \cong C$ both equal $I$.
\end{itemize}
\end{definition}

\noindent A coalgebra $C$ is called cocommutative if $\tau \circ \Delta = \Delta$.

\begin{definition}
A \emph{bialgebra} $H$ over $k$ is a $k$-vector space equipped with both an algebra $(H, \mu, \eta)$ and a coalgebra structure $(H, \Delta, \epsilon)$, such that the coproduct $\Delta: H \rightarrow H \otimes H$ and the counit $\epsilon: H \rightarrow k$ are algebra morphisms. These compatibility conditions can be expressed in terms of the following commutative diagrams\footnote{All diagrams were created using Paul Taylor's diagram package, available from \url{http://www.paultaylor.eu/diagrams/}}, where $\tau$ denotes the \emph{flip} operation $\tau(h_1,h_2) = (h_2,h_1)$:

\begin{multicols}{2}
\begin{center}
\begin{diagram}[labelstyle=\scriptstyle]
H^{\otimes 4} &&\rTo^{I \otimes \tau \otimes I}& &  H^{\otimes 4}\\
\uTo^{\Delta \otimes \Delta}    &&&&  \dTo_{\mu \otimes \mu}\\
H \otimes H&\rTo_{\mu}& H &\rTo_{\Delta}& H \otimes H
\end{diagram}

\begin{diagram}[labelstyle=\scriptstyle]
H \otimes H &\rTo^{\epsilon \otimes \epsilon}& k \otimes k\\ 
\dTo^{\mu} && \dTo_{\cong} \\
H &\rTo_{\epsilon}& k
\end{diagram}

\end{center}
\end{multicols}

\end{definition}

\noindent A bialgebra $H$ is called commutative if it is commutative as an algebra, and cocommutative if it is cocommutative as a coalgebra. 
\begin{remark}
There is symmetry in the definition of a bialgebra. Rather than requiring the coalgebra structure to respect the algebra structure in the above sense, we could have switched the role of the two structures. To complete the symmetry, we could in addition reverse the arrows in the two diagrams above. This would result in an equivalent definition.
\end{remark}

\paragraph{Grading.} Let $H$ be a \emph{graded} $k$-vector space, i.e. $H = \bigoplus_{n\geq 0}H_n$. There is a notion of a graded bialgebra, obtained by requiring the following of the algebra and coalgebra structure, respectively:

\begin{itemize}
\item[(i)] $\mu(H_p,H_q) \subset H_{p+q}$
\item[(ii)] $\Delta(H_n) \subset \bigoplus_{p+q=n} H_p \otimes H_q$.
\end{itemize}
The grading of an algebra $H$ gives rise to the \emph{grading operator} $Y: H \rightarrow H$ given by $$Y: h \mapsto \sum_{k\geq 0} kh_k,$$ where $h = \sum_{n \geq 0} h_n \in \bigoplus_{n \geq 0} H_n$. A graded bialgebra $H = \bigoplus_{n\geq 0} H_n$ is called \emph{connected} if $H_0 = k$. 

\begin{proposition}[{\cite{manchon06haf}}]
Let $H$ be a connected, graded bialgebra. Then, for any $x \in H_n$, $n \geq 0$, we have:
$$\Delta x = 1 \otimes x + x \otimes 1 + \tilde{\Delta} x, \text{  where  } \tilde{\Delta} x \in \bigoplus_{p+q = n, \,\, p,q > 0} H_p \otimes H_q.$$
\end{proposition}
\noindent We will often use the \emph{Sweedler notation} for the coproduct: 
$$\Delta x= \sum_{(x)} x_{(1)} \otimes x_{(2)}\quad \mbox{and} \quad \tilde{\Delta} x = \sum_{(x)} x' \otimes x''.$$

\begin{definition}
A \emph{Hopf algebra} is a bialgebra $(H,\mu,\eta,\Delta,\epsilon)$ together with an antihomomorphism $S$ on $H$, called the \emph{antipode}, with the property given by the commutativity of the following diagram: 
\begin{diagram}
& & H\otimes H & & \rTo^{\scriptstyle S\otimes 1} & & H\otimes H \\
& \ruTo^{\scriptstyle\Delta} & & & & & & \rdTo>{\scriptstyle\mu} \\
H & & \rTo^{\scriptstyle \varepsilon} & & k & & \rTo^{\scriptstyle \eta} & & H \\
& \rdTo<{\scriptstyle\Delta} & & & & & & \ruTo>{\scriptstyle \mu} \\
& & H\otimes H & & \rTo_{\scriptstyle 1\otimes S} & & H\otimes H \\
\end{diagram}
\end{definition}

\noindent A Hopf algebra is graded if it is graded as a bialgebra and the antipode satisifies $S(H_n) \subset H_n.$ If a bialgebra is graded and connected, then it is automatically a graded Hopf algebra:
\begin{proposition}[{\cite{manchon06haf}}]~\label{gradedantipode}
Any connected graded bialgebra is a Hopf algebra. The antipode $S$ is given recursively by $S(1) = 1$ and $$S(x) = -x - \sum_{(x)} S(x')x''$$ for $x \in \ker{\epsilon}$.
\end{proposition}

\subsection{Examples: The concatination and shuffle Hopf algebras}\label{ex:shufalg1}
Recurring in the sequel are Hopf algebras built from letters in an alphabet. We follow the notation of Reutenauer~\cite{reutenauer93fla}. Consider a finite or infinite alphabet  of letters $\A = \{a,b,c,\ldots\}$. We write $\A^*$ for the collection of all empty or non-empty words over $\A$, where $\one$ is the empty word. Let $k\langle \A \rangle$ be the $k$-algebra of non-commutative polynomials in $\A$. A polynomial $P\in k\langle \A \rangle$ will be written as a sum $$P = \sum_{\omega \in A^*} (P, \omega) \omega,$$ where $(P, \omega) \in k$ is non-zero only for a finite number of $\omega$. Let $P,Q\in k\langle \A \rangle$. The product of $P$ and $Q$, written as $PQ$, has coefficients $$(PQ, \omega) = \sum_{\omega = uv} (P, u) (Q, v).$$ 

The $k$-linear dual space  denoted $k\langle\langle\A\rangle\rangle:=\Hom_k(k\langle\A\rangle,k)$ is identified with all infinite $k$-linear combinations of words.  An $\alpha \in k\langle\langle\A\rangle\rangle$ can be written as an infinite series
\[\alpha = \sum_{\omega\in \A^*} (\alpha,\omega)\omega,\] where $(\alpha,\omega)\equiv \alpha(\omega)\in k$ and $(\cdot,\cdot)$ is the dual pairing defined such that words in $\A^*$ are orthogonal, $(\omega_1,\omega_2) = \delta_{\omega_1,\omega_2}$ for all $\omega_1,\omega_2\in \A^*$.

We define two different associative products on $k\langle\A\rangle$. The \emph{concatenation product} $\omega_1,\omega_2\mapsto \omega_1\omega_2$ obtained by concatination of words and the \emph{shuffle product} 
$\omega_1,\omega_2\mapsto \omega_1\sqcup\omega_2$ obtained by linearly combining all possible \emph{shuffles} of the two words i.e.\ combinations where the letters within each word are not internally permuted: $$abc \sqcup de = abcde + abdce + adbce + dabce + abdec + adbec + dabec + adebc + daebc + deabc.$$ The shuffle product can be defined recursively as
$$(a\omega_1) \sqcup (b\omega_2) = a(\omega_1 \sqcup b\omega_2) + b(a\omega_1 \sqcup \omega_2),$$ where $a,b \in \A$ and $\omega_1,\omega_2 \in \A^*$. The unit of both concatenation and shuffle is the empty word $\one$.

By dualization of these products we obtain the \emph{deconcatenation} and the \emph{deshuffle} coproducts. The deconcatination coproduct $\cpd\colon k\langle\A\rangle\rightarrow k\langle\A\rangle\tpr k\langle\A\rangle$ is defined for $\omega=a_1a_2\cdots a_k\in \A^*$ as:
\begin{equation}
\cpd(\omega) =  \sum_{i=1}^{k} a_1\cdots a_i\tpr a_{i+1}\cdots a_k .
\end{equation}
This coproduct is the dual of the concatenation product, so for any $P,Q\in k\langle\A\rangle$
$$(PQ,\omega) = ( P\tpr Q,\cpd(\omega)) = \sum_{(\omega)_{\cpd}}(P,\omega_{(1)})(Q,\omega_{(2)}) .$$
The deshuffle product $\cpds\colon k\langle\A\rangle\rightarrow k\langle\A\rangle\tpr k\langle\A\rangle$ is similarly defined such that
$$(P\sqcup Q,\omega) = ( P\tpr Q,\cpds(\omega)) = \sum_{{(\omega)}_{\cpds}}(P,\omega_{(1)})(Q,\omega_{(2)}) .$$
The two coproducts can also be characterized by requiring that the letters in the alphabet $\A$ are primitive, i.e. that $\Delta(a) = 1 \otimes a + a \otimes 1$ for $a \in \A$, and then extending $\Delta$ to be a homomorphism with respect to either of the two products on $k\langle\A\rangle$.
We refer to~\cite{reutenauer93fla} for explicit presentations of the deshuffle coproduct.

We remark that the vector space $k\langle\A\rangle$ can now be turned into Hopf algebras in two different ways. The cocommutative \emph{concatenation Hopf algebra} is obtained by taking the concatenation as product and the deshuffle as coproduct. The commutative \emph{shuffle Hopf algebra} $\Hsh(\A)$ is obtained by taking the shuffle as product and the deconcatenation as coproduct. Both these Hopf algebras share the same antipode:
\begin{equation}
S(a_1a_2\ldots a_k) = (-1)^k a_ka_{k-1}\ldots a_1,
\end{equation}
and in both cases the unit and counit is given by $\eta(1) = \one$ and $\epsilon(\one) = 1$, $\epsilon(\omega) = 0$ for all $\omega\in \A^*\backslash \one$. We write $\Hsh$ when $\A$ is understood.

The vector space $k\langle\A\rangle$ can be identified with the vector space underlying the tensor algebra $T(V)$ on the vector space $V$ generated by the alphabet $\A$. The two algebra structures (concatination and shuffling) correspond to the usual algebra structures given to the tensor algebra $T(V)$ and the tensor coalgebra $T^c(V)$, respectively.


\subsection{Characters and endomorphisms}

This section is based on \cite{ebrahimi-fard07alt}. See also \cite{manchon06haf}, \cite{burgunder2008eia} and \cite{patras02oda}.

Let $(H, \mu, \eta, \Delta, \epsilon)$ be a graded bialgebra, and $(A, \cdot, \eta_A)$ an algebra. The set $\Hom_k(H,A)$ of linear maps from $H$ to $A$ sending $\eta_H(1) =: 1_H$ to $\eta_A(1)=:1_A$ has an algebra structure given by the \emph{convolution product}: $$\alpha * \beta =  \mu_A \circ (\alpha \otimes \beta) \circ \Delta.$$ The convolutional unit is the composition of the counit of $H$ and the unit of $A$: $\delta := \eta_A \circ \epsilon$. The convolution can be written using the Sweedler notation: $$\alpha * \beta = \sum_{(x)} \alpha(x_{(1)}) \cdot \beta(x_{(2)}),$$ 
from which we find
$\alpha\ast \delta = \delta\ast\alpha = \alpha$. The \emph{unital algebra morphisms} from $H$ to $A$ consists of all $\alpha\in\Hom_k(H,A)$ such that
$\alpha(1_H) = 1_A$ and $\alpha(\mu(h, h')) = \alpha(h)\dpr\alpha(h')$ for all $h,h'\in H$.

\begin{proposition}[{\cite{manchon06haf}}]\label{A-valchar}
Let $H$ be a graded Hopf algebra and $A$ a commutative algebra. The set $\Hom_{Alg}(H,A)$ of unital algebra morphisms from $H$ to $A$ equipped with the convolution product, forms a group, $G(H,A),$ called the \emph{group of $A$-valued characters of $H$}. The inverse of an element $\alpha$ is given by $$\alpha^{*-1} = \alpha \circ S,$$ where $S$ is the antipode of $H$.
\end{proposition}
\noindent In the special case $A = k$, we get the group of characters of $H$, written as $G(H):=G(H,k)$.  The grading on $H$ splits the group of $A$-valued characters into graded components: $$G(H,A) \cong \prod_{n \geq 0} \Hom_{Alg}(H_n, A).$$ This is not a graded vector space, but rather the completion of one (see e.g. \cite{ebrahimi-fard07alt}), but we will still refer to it as a graded vector space. The restriction of a character $\alpha: H \rightarrow A$ to the degree $n$ component $H_n$ of $H$ will be denoted by $\alpha_n$.

\subsubsection{Infinitesimal characters, the exponential and the logarithm}

The \emph{infinitesimal $A$-valued characters}, written $\g(H,A)$ are the linear maps $\alpha$ from $H$ to $A$ such that: $$\alpha(\mu(h,h')) = \alpha(h)\cdot \delta(h') + \delta(h)\cdot \alpha(h'),$$ where $\delta=\eta_A \circ \epsilon$. This is a Lie algebra under the bracket induced by the convolution product: $[\alpha, \beta] = \alpha * \beta - \beta * \alpha.$ 
In the special case where $A=k$ we write $\g(H)$ for $\g(H,k)$.

The characters and the infinitesimal characters are connected via the \emph{exponential} and the \emph{logarithmic} map. For $\alpha \in \Hom_k(H,A)$, the exponential and logarithm with respect to convolution are given by the formal series: 
\begin{eqnarray*}
\exp^*(\alpha) &=& \sum_{n \geq 0} \frac{1}{n!} \alpha^{*n}\\
\log^*(\delta+\alpha) &=& \sum_{n \geq 1} \frac{(-1)^{n-1}}{n} \alpha^{*n}.
\end{eqnarray*}
If $H$ is graded and connected, and if $\alpha(1) = 0$, where $1 \in k = H_0$, then  $\alpha^{*k} = \alpha \ast \cdots\ast \alpha = 0$ on $H_n$ for $n < k$, and therefore both these sums are finite when restricted to $H_n$. The maps $\exp^*$ and $\log^*$ give a bijection between $G(H,A)$ and $\g(H,A)$.

\begin{example}\label{ex:shufalg2}Let $\Hsh$ denote the shuffle algebra over $\A$. Consider the dual space $k\langle\langle\A\rangle\rangle$ equipped with the convolution product
\[(\alpha\ast\beta,\omega) = \sum_{(\omega)_{\cpd}} (\alpha,\omega_{(1)})(\beta,\omega_{(2)}) = \sum_{\omega=\omega_1\omega_2}(\alpha,\omega_1)(\beta,\omega_2).\]
Note that convolution is just concatenation of series $\alpha\ast\beta = \alpha\beta$.
The characters and infinitesimal characters $\g(\Hsh),G(\Hsh)\subset k\langle\langle\A\rangle\rangle$ are given as
\begin{eqnarray*}
\g(\Hsh) & = & \stset{\alpha\in k\langle\langle\A\rangle\rangle}{\alpha(\one) = 0 \mbox{ and } \alpha(\omega_1\sqcup \omega_2) = 0 \mbox{ for all } \omega_1,\omega_2\in \A^*\backslash \one}\\
G(\Hsh) & = & \stset{\alpha\in k\langle\langle\A\rangle\rangle}{\alpha(\one) = 1 \mbox{ and } \alpha(\omega_1\sqcup \omega_2) = \alpha(\omega_1)\alpha(\omega_2) \mbox{ for all } \omega_1,\omega_2\in \A^*\backslash \one} .
\end{eqnarray*}
The convolutional unit $\delta$ is given as $(\delta,\one) = 1$ and $(\delta,\omega) = 0$ for all $\omega\in\A\backslash \one$.
The logarithm of $\alpha\in G(\Hsh)$ can be computed as $\log(\alpha) = \sum_{n>0} \frac{(-1)^{n-1}}{n} (\alpha-\delta)^{\ast n}$. For any $\omega\in \A^*$
we find that $(\log(\alpha),\omega)$ is given by a finite sum expressed in terms of the Eulerian idempotent.
\end{example}

\subsubsection{Eulerian idempotent}\label{eulidem}
Let $H$ be a commutative, connected and graded Hopf algebra. Consider $\End_k(H) = \Hom_k(H,H)$ equipped with the convolution product $\ast$. Let $\id\in\End_{k}(H)$ be the identity endomorphism and $\delta = \eta\circ \epsilon\in\End_{k}(H)$ the unit of convolution.

\begin{definition}[{\cite{loday97ch}}]
The Eulerian idempotent $e\in \End(H)$ is given by the formal power series $$e := \log^*(\Id) = J - \frac{J^{*2}}{2} + \frac{J^{*3}}{3} + \cdots (-1)^{i+1} \frac{J^{*i}}{i} + \cdots,$$ where $J=\Id-\delta.$
\end{definition}

\begin{proposition}[{\cite{loday97ch}}]
For any commutative graded Hopf algebra $H$, the element $e \in \End_k(H)$ defined above is an idempotent: $e \circ e = e$.
\end{proposition}
\noindent The practical importance of the Eulerian idempotent  in numerical analysis arises in backward error analysis, where the following lemma provides a computational formula for the logarithm:

\begin{proposition} For $\alpha\in G(H)$ and $h\in H$, we have \[\log^\ast(\alpha)(h) = \alpha(e(h)).\] In other words, the logarithm can be written as right composition with the eulerian idempotent: $$\log^{\ast} = \_ \circ e: G(H) \rightarrow \g(H).$$
\end{proposition}

\noindent The result follows from the following computation, which uses that $\alpha$ is a homomorphism:
\[((\alpha-\delta)^{\ast l},\omega) = \mu^l_k\opr (\alpha\tpr\cdots\tpr\alpha)\opr \tilde{\Delta}^l\omega = \alpha\opr\mu^l_H\opr \tilde{\Delta}^l\omega = \alpha\opr J^{\ast l}\omega,\]
where $(-)^l$ denotes $l$-fold application.

\subsubsection{The graded Dynkin operator}\label{dynkinop}
There is another  bijection between the infinitesimal characters and the characters in any commutative graded Hopf algebra $H$, described in \cite{ebrahimi-fard07alt}. The bijection is given in terms of the \emph{Dynkin operator} $D: H \rightarrow H$. 

Classically, the Dynkin operator is a map $D: k\langle\A\rangle \rightarrow Lie(\A)$, where $Lie(\A)=\g(\Hsh)\cap k\langle\A\rangle$ are the Lie polynomials. The classical Dynkin operator is given by left-to-right bracketing: $$D(a_1...a_n) = [\dots[[a_{1},a_{2}], a_{3}], \dots, a_{n}], \quad\mbox{where $[a_i,a_j] = a_ia_j-a_ja_i$}.$$ 
Letting $Y(\omega) = \#(\omega)\omega$ denote grading operator, where $\#(\omega)$ is word length, it is known that the \emph{Dynkin idempotent}, given as $Y^{-1}D$, is an idempotent projection on the subspace of Lie polynomials.
As in \cite{ebrahimi-fard07alt}, the Dynkin operator can be written as the convolution of the antipode $S$ and the grading operator $D = S\ast Y$. This description can be generalized to any graded, connected and commutative Hopf algebra $H$:

\begin{definition}
Let $H$ be a graded, commutative and connected Hopf algebra with grading operator $Y: H \rightarrow H$.  The \emph{Dynkin operator} is the map $D: H \rightarrow H$ given as $$D := S * Y.$$
\end{definition}

\begin{lemma}[{\cite{ebrahimi-fard07alt}}]
The Dynkin operator is a $H$-valued infinitesimal character of $H$.
\end{lemma}

\begin{theorem}[{\cite{ebrahimi-fard07alt}}]\label{th:dynkinvers}
Right composition with the Dynkin operator induces a bijection between $G(H)$ and $\g(H)$: $$\_ \circ D: G(H) \rightarrow \g(H).$$ The inverse is given by $\Gamma: \g(H) \rightarrow G(H)$ as 
\begin{equation}\label{eq:dynkininverse}
\Gamma(\alpha) = \sum_{n} \sum_{\begin{array}{cc} k_1 + \cdots + k_l = n, \\ k_1,..., k_l > 0 \end{array}} \frac{\alpha_{k_1} * \cdots * \alpha_{k_l}}{k_1(k_1+k_2)\cdots (k_1 + \cdots + k_l)},\end{equation}
where $\alpha_k = \left.\alpha\right|_{H_k}$.
\end{theorem}

\noindent Later we will apply the Dynkin operator and its inverse in the setting of a shuffle algebra $\Hsh(\OT)$, where $\OT$ is an alphabet of all ordered rooted trees, and the grading $|\tau|$ of $\tau\in \OT$ counts the nodes in the tree.

\section{Algebras of formal diffeomorphisms on manifolds}

The main goal of this section is to arrive at Lie--Butcher series and the underlying Hopf algebra $\Hn$. This Hopf algebra contains the Connes--Kreimer Hopf algebra as a subalgebra and is also closely related to $\Hsh$. To emphasize the natural connection between Lie--Butcher series, $\Hn$ and more classical Lie series, we start with a discussions of Lie series (autonomous and non-autonomous).

\subsection{Autonomous Lie series}\label{sec:autlie}
In this section we review the well-known theory of Lie series on manifolds and the corresponding Hopf algebraic structures of the free Lie algebra. The algebraic theory is detailed in~\cite{reutenauer93fla,cartier2006apo} and for the 
analytical theory we refer to~\cite{abraham88mta}.

Let $F$ be a vector field on a manifold $\M$ and $\Phi_{t,F}\colon \M\rightarrow\M$ its flow.  Let $\psi\colon \M\rightarrow E$ be a section of a vector bundle over $\M$, and let $\Phi_{t,F}^* \psi$ denote the pullback.
For the applications later in this paper we will only consider trivial bundles, in which case we write $\psi\colon \M\rightarrow \V$ for some vector space $\V$ and define pullback as composition $\Phi_{t,F}^*\psi = \psi\opr \Phi_{t,F}$. 
The \emph{Lie derivative} of $\psi$ is defined as
\begin{equation}\label{eq:liederiv}F[\psi]  = \left.\frac{\partial}{\partial t}\right|_{t=0}\Phi_{t,F}^*\psi .
\end{equation}
Composition of Lie derivatives defines an associative, non-commutative product of vector fields $F,G\mapsto F G$, where vector fields are first order differential operators. The product $F G$ is the second order differential operator $(F G)[\psi] = F[G[\psi]]$ etc. We let $\one$ denote the 0th order identity operator $\one[\psi] = \psi$. 
The linear span of all differential operators of all orders forms the universal enveloping algebra $U(\XM)$.

The basic pullback formula is (\cite{abraham88mta}):
\begin{equation}\label{eq:pullback}\frac{\partial}{\partial t} \Phi_{t,F}^*\psi =  \Phi_{t,F}^*(F[\psi]) .
\end{equation}
Iterating this we find ${\left.\partial^n/\partial t^n\right|}_{t=0}\Phi_{t,F}^*\psi = F[F[\cdots [\psi]]] := F^n[\psi]$, and hence follows the (Taylor)--Lie form of a \emph{pullback series}:
\[\Phi_{t,F}^*[\psi] = \sum_{j=0}^\infty \frac{t^j}{j!}F^j[\psi] := \mbox{Exp}(tF)[\psi] .\]

\noindent Fundamental questions are: Which series in $U(\XM)$ represent vector fields and which represent pullback series? How do we algebraically characterize compositions and the inverse of pullback series? How do we understand the $\Exp$ map taking vector fields to their pullback series, and what about the inverse $\Log$ operation? These questions are elegantly answered in terms of the shuffle Hopf algebra. 
We will detail these issues, and see that the same structures reappear in the discussion of B-series later.

An algebraic abstraction of Lie series starts with 
fixing a (finite or infinite) alphabet $\A$ and a map $\nu\colon \A\rightarrow \XM$ assigning each letter to a vector field.
As in Example~\ref{ex:shufalg1} we let $\RA$
denote all finite $\RR$-linear combinations of words built from $\A$ and $\Hsh$ the shuffle algebra. The map $\nu$ can be uniquely extended to a linear $\F_\nu\colon \RA\rightarrow U(\XM)$  as a concatenation homomorphism:
\begin{eqnarray*}
\F_\nu(\one) & = & \one,\\
\F_\nu(a) & = & \nu(a) \quad \mbox{ for all letters $a\in \A$},\\
\F_\nu(\omega_1\omega_2)& =& \F_\nu(\omega_1)\F_\nu(\omega_2) \quad \mbox{for all words $\omega_1,\omega_2\in \A^*$.}
\end{eqnarray*}
We extend $\F_\nu$ to a map $\Bs_t$ taking an infinite series $\alpha\in\RR\langle\langle\A\rangle\rangle$ to an infinite formal series 
$\Bs_{t}(\alpha)\in U(\XM)^*$, defined for $t\in\RR$ as follows: Consider the alphabet $\A$ with a grading 
$|a|\in \NN^+$ for all $a\in \A$. This extends to $\Hsh$ as $|\omega| = |a_1|+\ldots+|a_k|$ for all $\omega=a_1\ldots a_k\in \A^*$, $|\one| = 0$, thus $\Hsh$ becomes a graded connected Hopf algebra. Given the grading we define
\begin{equation}\label{eq:btseries}
\Bs_{t}(\alpha) = \sum_{\omega\in \A^*} t^{|\omega|} \alpha(\omega)\F_\nu(\omega). 
\end{equation}

\noindent Consider $\Hsh^*=\RR\langle\langle\A\rangle\rangle$ with the convolution $\alpha\ast \beta = \alpha\beta$ as in Example~\ref{ex:shufalg2}. By construction $\Bs_t$ is a convolution homomorphism, \[\Bs_t(\alpha\ast\beta) = \Bs_t(\alpha)\Bs_t(\beta).\]
For a real valued infinitesimal character $\alpha\in\g(\Hsh)$, and a fixed $t=h$,  $\Bs_h(\alpha)$ is a formal vector field on $\M$. For a real valued character $\beta\in G(\Hsh)$, $\Bs_h(\beta)$ represents a formal diffeomorphism $\Phi_h$ on $\M$ via the pullback series \[\Bs_h(\beta)[\psi] = \psi\opr \Phi_h\quad \mbox{for $\psi\colon \M\rightarrow \RR$.}\] Note, however, that pullbacks compose contravariantly with respect to composition of diffeomorphisms:
\[\Bs_h(\beta_1\ast\beta_2)[\psi] = \Bs_h(\beta_1)\Bs_h(\beta_2)[\psi] = \psi\opr \Phi_2\opr\Phi_1.\]

To summarize: Composition of diffeomorphisms is modelled by convolution in $G(\Hsh)$ (in opposite order), the inverse of a diffeomorphism is computed by right composing with the antipode, the convolutional exponential maps to the exponential of Lie series and the logarithm is computed by composing with the Eulerian idempotent.
\begin{eqnarray*}
\Bs_h(\beta\opr S)\Bs_h(\beta) &=& \one \quad \mbox{ for $\beta\in G(\Hsh)$}\\
\Bs_h(\exp^*(\alpha))& =& \mbox{Exp}(\Bs_h(\alpha))\quad \mbox{ for $\alpha\in \g(\Hsh)$}\\
\Bs_h(\beta) & = & \mbox{Exp}(\Bs_h(\beta\opr e))\quad \mbox{ for $\beta\in G(\Hsh)$}.
\end{eqnarray*}
In the next section  we discuss flows of non-autonomous equations, and we will see that right composition with the Dynkin idempotent represents algebraically the operation of finding a non-autonomous vector field corresponding to a diffeomorphism on a manifold. 

\begin{remark}In~\cite{munthe-kaas99cia} the Lie algebra of infinitesimal characters $\g(\Hsh)$ is studied as a graded free Lie algebra. An explicit formula for the dimension of the homogeneous components $\g_k = \left.\g(\Hsh)\right|_k$ is derived for general gradings. This is very useful for the study of the complexity of Lie group integrators.
\end{remark}


\subsection{Time-dependent Lie series}\label{sec:tlie}

The classical Fa\`a di Bruno Hopf algebra models the composition of formal diffeomorphisms on $\RR$ (\cite{figueroa2005fdb}, \cite{figueroa2005cha}, \cite{foissy2008fdb}). We will see that this has a natural generalization to compositions of time-dependent flows on manifolds. We introduce a \emph{Dynkin--Fa\`a di Bruno bialgebra} describing the composition of flows of time-dependent vector fields on a coarse level that considers only the grading of the terms in the $t$-expansion of the time-dependent vector fields.

\subsubsection{Non-commutative Bell polynomials and Dynkin--Fa\`a di Bruno bi-algebra}

Let $\I=\{d_j\}_{j=1}^\infty$ be an infinite alphabet in 1--1 correspondence with $\NN^+$, and consider the 
free associative algebra $\Hfdb = \RR\langle \I \rangle$ with the grading given by $|d_j| = j$ and $|d_{j_1}\cdots d_{j_k}|= j_1+\cdots+j_k$. Let $\partial\colon \Hfdb\rightarrow \Hfdb$ be the derivation given by $\partial(d_i) = d_{i+1}$, linearity and the Leibniz rule
$\partial(\omega_1\omega_2) = \partial(\omega_1)\omega_2 + \omega_1\partial(\omega_2)$ for all $\omega_1,\omega_2\in\I^*$. We let $\#(\omega)$ denote the length of the word $\omega$.

\begin{definition}\label{def:bell}The non-commutative Bell polynomials $B_n \equiv B_n(d_1,\ldots, d_n)\in \RR\langle \I \rangle$ are defined by the recursion
\begin{eqnarray*}
B_0 & = & \one\\
B_n& =& (d_1+\partial)B_{n-1} = (d_1+\partial)^n \one\quad \mbox{for $n>0$}.
\end{eqnarray*}
\end{definition}

The first of these are given as
 \begin{eqnarray*}
B_0 &=& \one\\
B_1 & = & d_1\\
B_2 & = & d_1^2 + d_2 \\
B_3 &=& d_1^3 + 2d_1d_2 + d_2d_1+d_3\\
B_4 & = & d_1^4+ 3d_1^2 d_2 +2d_1d_2d_1 + d_2d_1^2 +3d_1d_3+  d_3d_1 + 3d_2d_2 + d_4.
\end{eqnarray*}
The polynomials $B_n$  are introduced in~\cite{munthe-kaas95lbt,munthe-kaas98rkm} to explain the Butcher order theory of Runge--Kutta methods in a manifold context, and generalize to certain classes of numerical integrators on manifolds.

\begin{remark} Additional insight to the Bell polynomials are obtained by  considering the free associative algebra generated by two symbols $d_1$ and $\partial$, defining \[d_i := [\partial, d_{i-1}]= \partial d_{i-1} - d_{i-1}\partial\quad\mbox{ for $i>1$.}\]
We find by induction that  $(d_1+\partial)^n$ satisfies the binomial relation
\begin{equation}
\left(d_1+\partial\right)^n  =  \sum_{k=0}^n \binom{n}{k} B_{k}(d_1,\ldots,d_k)\partial^{n-k},
\end{equation}
which yields the formula
\begin{equation}
\label{eq:expsplit}
\exp\left(d_1+\partial\right) = \sum_{m=0}^\infty\frac{B_m(d_1,\ldots,d_m)}{m!} \exp\left(\partial\right),
\end{equation}
and also the recursion
\begin{equation}
B_{n+1}(d_1,\ldots,d_{n+1}) = \sum_{k=0}^n\binom{n}{k}B_k(d_1,\ldots,d_k)d_{n-k+1} \quad\mbox{for $n>0$.}
\end{equation}
\end{remark}

The  non-commutative  \emph{partial Bell polynomials} $B_{n,k}\equiv B_{n,k}(d_1,\ldots,d_{n-k+1})$ are defined as the part of $B_n$ consisting of the words $\omega$ of length $\#(\omega)=k>0$, 
e.g. $B_{4,3} = 3d_1^2 d_2 +2d_1d_2d_1 + d_2d_1^2$. Thus
\[B_n = \sum_{k=1}^n  B_{n,k} .\]
A bit of combinatorics yields an explicit formula:
  \begin{equation}
B_{n,k} = \mathop{\sum_{\omega\in\I^*}}_{|\omega| = n, \#(\omega)=k} \kappa({\omega})\binom{n}{\omega}
\omega ,
\end{equation}
where for $\omega = d_{j_1}d_{j_2}\cdots d_{j_k}$
\[ \binom{n}{\omega}\equiv\binom{n}{|d_{j_1}|,|d_{j_2}|,\ldots,|d_{j_k}|} :=    \frac{n!}{j_1!j_2!\cdots j_k !}\]
are the multinomial coefficients and the coefficients $\kappa(\omega)$ are defined as
  \begin{equation}\kappa(\omega)  \equiv \kappa(|d_{j_1}|,|d_{j_2}|,\ldots,|d_{j_k}|) := \frac{j_1j_2\cdots j_k}{j_1(j_1+j_2)\cdots(j_1+j_2+\cdots+j_k)}.\end{equation}
The coefficients $\kappa$ form a partition of unity on the symmetric group $S_k$, 
\[\sum_{\sigma\in S_k} \kappa(\sigma(\omega)) = 1,\]
where $\sigma(\omega)$ denotes a permutation of the letters in $\omega$. E.g.\ $\kappa(1,2)+\kappa(2,1)=\frac23+\frac13=1$.

It is often useful to employ polynomials  $Q_{n}$ and $Q_{n,k}$ related to  $B_{n}$ and $B_{n,k}$ by the following rescaling:
\begin{eqnarray}
Q_{n,k}(d_1,\ldots,d_{n-k+1}) &=& \frac{1}{n!}B_{n,k}(1!d_1,\ldots,j!d_j,\ldots) = \mathop{\sum}_{|\omega| = n, \#(\omega)=k} \kappa({\omega})
\omega\\
Q_n(d_1,\ldots,d_n)&=&\sum_{k=1}^n Q_{n,k}(d_1,\ldots,d_{n-k+1})\\
Q_0 &:=& \one.
\end{eqnarray}

\noindent Note that $B_n$ and $B_{n,k}$ become the classical Bell- and partial Bell polynomials when the product in $\RR\langle \I \rangle$ is commutative, i.e. in the free commutative algebra on $\I$. A non-commutative Fa\`a di Bruno Hopf algebra is studied in~\cite{brouder06nch}. However, their definition differs from the present by defining the polynomials $Q_{n,k}$ without the factor $\kappa$ that associates  different factors to different permutations of a word (adding up to 1 over all permuatations).

These Bell polynomials are closely related to the graded Dynkin operator on a connected graded Hopf algebra $H$.
For $\alpha\in  H^*$, define a graded algebra homomorphism $d_i\mapsto d_i(\alpha)\colon \Hfdb\rightarrow H^*$ as
\begin{equation}
d_i(\alpha) = \alpha_i = \left.\alpha\right|_{H_i}, \qquad d_id_j(\alpha) = \alpha_i\ast\alpha_j .
\end{equation}

\begin{proposition} \label{prop:dynkinIdem}The operator  defined as
\begin{equation}\label{eq:Q}
Q(\alpha) = \sum_{n=0}^\infty Q_n(\alpha),
\end{equation}
is a bijection from infinitesimal characters to characters $Q\colon \g(H)\rightarrow G(H)$ with inverse given by right composition with the Dynkin idempotent $Y^{-1}\opr D$,
\begin{equation}
Q^{-1}(\beta) = \beta\opr Y^{-1}\opr D,
\end{equation}
where $Y$ is the grading operator on $H$ and $D=S\ast Y$ is the graded Dynkin operator.
\end{proposition}

\begin{proof}For $\alpha\in \g(H)$ we have
\begin{equation}\label{eq:QGamma}\Gamma(\alpha\opr Y) = \sum_{n=0}^\infty\sum_{j_1+\cdots+j_k=n} \frac{j_1j_2\cdots j_k}{j_1(j_1+j_2)\cdots(j_1+\cdots+j_k)}\alpha_{j_1}\ast \cdots\ast\alpha_{j_k}
= Q(\alpha) ,\end{equation}
thus the result follows from Theorem~\ref{th:dynkinvers}.
\end{proof}

The non-commutative Dynkin--Fa\`a di Bruno bialgebra $\Hfdb$ is obtained by taking the algebra structure of $\Hfdb$ and defining the coproduct $\cpfdb$ as
\begin{eqnarray}\label{eq:fdb1}
\cpfdb(\one)  & = & \one\tpr \one\\
\cpfdb(d_n) & = & \sum_{k=1}^n B_{n,k}\tpr d_k .\label{eq:fdb2}
\end{eqnarray}
This extends to all of $\Hfdb$ by the product rule $\cpfdb(d_i d_j) = \cpfdb(d_i)\cpfdb(d_j)$. Thus, e.g.\
\begin{eqnarray*}
\cpfdb(d_1) & = & d_1\tpr d_1\\ 
\cpfdb(d_2) & = & d_1^2\tpr d_2 + d_2\tpr d_1\\
\cpfdb(d_1d_2) & = & d_1^3\tpr d_1d_2 + d_1d_2\tpr d_1^2.
\end{eqnarray*}
Note that the coproduct is \emph{not} graded by $|\cdot|$, thus Proposition~\ref{gradedantipode} does not hold for $\Hfdb$. By a lengthy (but not enlightening) induction argument we can prove:

\begin{lemma}~\label{coprod of Bell} The coproduct of the partial Bell polynomials are given as
\begin{equation}\label{eq:fdb3}
\cpfdb(B_{n,k})  =  \sum_{\ell=1}^n B_{n,\ell}\tpr B_{\ell,k} .
\end{equation}
\end{lemma}
\noindent Note that  $B_{n,1}=d_n$, thus (\ref{eq:fdb2}) is a special case of (\ref{eq:fdb3}). Summing the partial $B_{n,k}$ over $k$, we find the coproduct of the full Bell polynomials:
\[\cpfdb(B_{n})  =  \sum_{k=1}^n B_{n,k}\tpr B_{k}.\]
\noindent Using Lemma \ref{coprod of Bell} and the fact that $B_{n,k} = 0$ for $k > n$, one can easily show that $\Hfdb$ is a bialgebra.

\begin{proposition} $\Hfdb = \RR\langle \I \rangle$ with the non-commutative concatenation product and the coproduct $\cpfdb$ form a bialgebra $\Hfdb$ which is neither commutative nor cocommutative. 
\end{proposition}

\subsubsection{Pullback along time-dependent flows}
Let $F_t = \sum_{j=0}^\infty F_{j+1} \frac{t^{j}}{j!}$ be a time-dependent vector field on $\M$ where $F_{j}= \left.F_t^{(j-1)}\right|_{t=0}$. Let $\Phi_{t,F_t}$ be the solution operator of the corresponding non-autonomous equation, such that
\[y(t) = \Phi_{t,F_t}y_0 \quad\mbox{solves}\quad y'(t) = F_t(y(t)),\quad y(0)=y_0.\]
Note that $\Phi_{t,F_t}$ is \emph{not} a 1-parameter subgroup of diffeomorphisms in $t$.

\begin{lemma}\cite{munthe-kaas95lbt} The $n$-th time derivative of the pullback of a (time-independent) function $\psi$ along the time-dependent flow $\Phi_{t,F_t}$ is given as
\begin{equation}\label{eq:Bnderiv1} \frac{\partial^n}{\partial t^n}\Phi_{t,F_t}^*\psi = B_n(F_t)[\psi] ,\end{equation}
where $B_n(F_t)$ is the image of $B_n$ under the homomorphism from $\Hfdb$ to $U(\XM)$ given by $d_i\mapsto F_t^{(i-1)}$. In particular
\begin{equation}\left.\label{eq:Bnderiv} \frac{\partial^n}{\partial t^n}\right|_{t=0}\Phi_{t,F_t}^*\psi = B_n(F_1,\ldots,F_n)[\psi] .\end{equation}
\end{lemma}

\begin{proof} The non-autonomous vector field $F_t$ on $\M$ corresponds to the autonomous field $F_t+\partial/\partial t$ on $\M\xpr \RR$, thus (\ref{eq:pullback}) yields
\[\frac{\partial}{\partial t} \Phi_{t,F_t}^*\psi =  \Phi_{t,F_t}^*\left((F_t+\partial/\partial t)[\psi]\right) 
\Rightarrow \frac{\partial^n}{\partial t^n} \Phi_{t,F_t}^*\psi =  \Phi_{t,F_t}^*\left((F_t+\partial/\partial t)^n[\psi]\right)
.\]
Consider the homomorphism induced from $d_1\mapsto F_t$ and $\partial \mapsto \partial/\partial t$, thus $d_i\mapsto F_t^{(i-1)}$. 
Equation~(\ref{eq:Bnderiv1}) follows directly from Definition~\ref{def:bell}. At $t=0$ we have
$d_i\mapsto F_i$, thus (\ref{eq:Bnderiv}).
\end{proof}

\begin{remark}
Note that~(\ref{eq:expsplit}) yields a space-time split formula for pullback which is valid also for pullback of a time-dependent function $\psi_t$. The pullback for $t\in[0,h]$ developed at $t=0$ becomes
\[\Phi_{h,F_t}^*\psi_t = \left.\exp\left(h(F_t+\frac{\partial}{\partial t})\right)\right|_{t=0}\!\!\!\!\!\!\!\!\![\psi_t] = \left.\sum_{n=0}^\infty\frac{h^n}{n!}B_n(F_t)\exp(h\frac{\partial}{\partial t})\right|_{t=0}\!\!\!\!\!\!\!\!\![\psi_t] 
= \sum_{n=0}^\infty\frac{h^n}{n!}B_n(F_1,\ldots,F_n)[\psi_h] . \]
\end{remark}

The Dynkin idempotent relates pullback series with their corresponding time-dependent vector fields. Let $\A$ be an arbitrary alphabet with a grading $|\cdot|\colon \A\rightarrow \NN^+$, let $\Hsh = \Hsh(\A)$ be the corresponding graded shuffle algebra and let $\Bs_t(\alpha)$ be as in (\ref{eq:btseries}).

\begin{proposition} Let $\alpha\in \g(\Hsh)$ and $\beta=Q(\alpha) \in G(\Hsh)$ be related by the graded Dynkin idempotent as
in Proposition~\ref{prop:dynkinIdem}. Define the time-dependent vector field
\[F_t =\frac{\partial}{\partial t} \Bs_t(\alpha).\]
Then pullback of a time-independent $\psi$ along the time-dependent flow $\Phi_{t,F_t}$ is given as
\begin{equation}\label{eq:bspullback}\Phi_{t,F_t}^*\psi = \Bs_t(\beta)[\psi]. \end{equation}
\end{proposition}

\begin{proof}  We have $F_t = \sum_{j=0}^\infty F_{j+1} \frac{t^{j}}{j!}$ where $F_j = \F_\nu(j!\alpha_j)$.
Developing the Taylor series of $\Phi_{t,F_t}^* \psi$ at $t=0$ we get from~(\ref{eq:Bnderiv})
\[\Phi_{t,F_t}^* \psi =\sum_{n=0}^\infty\frac{t^n}{n!}B_n(F_1,\ldots,F_n)[\psi] .\]
Thus
\[\frac{1}{n!}B_n(F_1,\ldots,F_n) = \F(\frac{1}{n!}B_n(1!\alpha_1,\ldots,n!\alpha_n) ) = \F(Q_n(\alpha_1,\ldots,\alpha_n)).\]
Using~(\ref{eq:QGamma}) we obtain the result.
\end{proof}

\subsection{Lie--Butcher theory}\label{sec:LB}
Pullback formulas such as (\ref{eq:bspullback}) relate the time derivatives of $F_t$ with the spatial derivatives of a function $\psi$. We have captured the algebraic structure of the temporal derivations through the Dynkin idempotent $Y^{-1}\opr D\colon G(\Hsh)\rightarrow \g(\Hsh)$ and its inverse $\Gamma\opr Y \colon\g(\Hsh)\rightarrow\G(\Hsh)$. However, the spatial Lie derivation $\Bs_t(\beta)[\psi]$ cannot be algebraically characterized within this structure.  In order to do this,  we need to refine the Hopf algebra $\Hsh$. On the manifold $M$, we obtain a refined version of $U(\XM)$ by expanding differential operators in terms of a non-commuting frame on $\XM$.  If the manifold is $\RR^n$ and the frame is the standard commutative coordinate frame, the construction yields the classical Butcher formulation and the Connes--Kreimer Hopf algebra \cite{brouder2000rkm}. More generally we obtain a Hopf algebra $\Hn$,  built on forests of planar trees, which contains the Connes--Kreimer algebra as a subalgebra. In $\Hn$ we can represent Lie derivation in terms of tree graftings.

\subsubsection{Differential operators in $U(\XM)$  expanded in a non-commuting frame}\label{sec:noncomm}
Let $\XM$ denote the Lie algebra of all vector fields on $\M$ and let $\g\subset\XM$ be a transitive Lie subalgebra, in the sense that $\g$ everywhere spans $T\M$. This means that $\g$ defines a frame on the tangent bundle. We do not assume that the frame forms a basis. In general $\dim(\g)\geq\dim(\M)$, and in case of strict inequality we have a non-trivial isotropy subgroup at any point.

Let $U(\g)$ denote the universal enveloping algebra of $\g$. We let $\g^\M$ and $U(\g)^\M$ denote maps from $\M$ to $\g$ and from $\M$ to $U(\g)$. Since $\g$ is assumed to be transitive, we can represent any vector field $F\in \XM$ with a function $f\in \g^\M$ as in Section~\ref{sec:eeuler}. Similarly, any higher order differential operator in $U(\XM)$ can be represented as a function in $U(\g)^\M$. We have the natural inclusion $\g\subset \gM$ and $U(\g)\subset\UgM$ as constant maps, called frozen vector fields and higher order differential operators. We identify $U(\g)^\M$ with sections of the trivial vector bundle $\M\tpr U(\g)\rightarrow \M$, and for a diffeomorphism $\Phi\colon\M\rightarrow \M$ we define pullback of $f\in U(\g)^\M$ as $\Phi^*f = f\opr \Phi\in U(\g)^\M$. Pullback in this bundle defines a parallel transport which gives rise to a flat connection with torsion. For  $f,g\in \UgM$ we define the connection $f[g]\in\UgM$ pointwise from the Lie derivative as
\begin{equation*}
 \label{eq:sectionderiv}
 f[g](p) = \left(f(p)[g]\right)(p), \qquad p\in\M .
\end{equation*}
Similarly, the concatenation in $U(\g)$ is extended pointwise to a concatenation product $fg\in \UgM$ as
\begin{equation*}
 (fg)(p) = f(p)g(p), \qquad p\in\M .
\end{equation*}
This is called the \emph{frozen composition} of $f$ and $g$. We can also compose $f$ and $g$ as non-frozen differential operators $f\bpr g\in \UgM$:
\begin{equation*}
 (f\bpr g)[h] = f[g[h]], \qquad \mbox{for all $h\in \UgM$ .}
\end{equation*}
This is identical to the composition in $U(\XM)$, which in Section~\ref{sec:autlie} was written as $F,G\mapsto FG$ for 
$F,G\in \XM$.

It might be illustrative to write out the operations explicitly in terms of a basis $\{\partial_k\}_{k=1}^n$ of (non-commuting) vector fields spanning $\g$. Writing $f,g\in \g^\M$ in terms of the frame as $f = \sum_k f_k \partial_k$ and $g = \sum_\ell g_\ell\partial_\ell$ for $f_k,g_\ell \in \RR^\M$, we have
\begin{eqnarray*}
fg & = & \sum_{k,\ell}f_kg_\ell \partial_k\partial_\ell\\
f[g] & = & \sum_{k,\ell}f_k\partial_k[g_\ell] \partial_\ell\\
f\bpr g & = & \sum_{k,\ell}f_k\partial_k[g_\ell] \partial_\ell+\sum_{k,\ell}f_kg_\ell \partial_k\partial_\ell.
\end{eqnarray*} 
\noindent The connection $f[g]$, the frozen composition $fg$ and nonfrozen composition $f\bpr g$ are related as:

\begin{lemma}
 \label{lem:dalgebra}
 Let $f\in\gM$ and $g,h\in \UgM$. Then we have
 \begin{align*}
   \one[g] & = g\\
   f[gh] & = f[g]h + g(f[h]), \quad \mbox{(Leibniz)}\\
   (f\bpr g)[h] := f[g[h]] & = (fg)[h] + (f[g])[h],
 \end{align*}
 where $\one\in \UgM$ is the constant identity map.
\end{lemma}
\noindent The proof is given in~\cite{munthe-kaas08oth}. Note the difference between $fg$ and $f\bpr g$. In the concatenation the value of $g$ is frozen to ${g}(p)$ before the differentiation with $f$ is done, whereas in the latter case the spatial variation of $g$ is seen by the differentiation using $f$. Interestingly, the work of Cayley from 1857~\cite{cayley1857taf} starts with the same result for vector fields expanded in the commuting frame $\partial/\partial x_i$. 

From this lemma we may compute the torsion and curvature of the connection. Let $f,g\in \g^\M$. We henceforth let $[f,g]_\bpr := f\bpr g - g\bpr f$ denote the Jacobi bracket and $[f,g]=fg-gf$ the \emph{frozen bracket}. The frozen bracket is computed pointwise from the bracket in $\g$  as $[f,g](p)= [f(p),g(p)]_\g$.  Writing the connection as $\nabla_fg:=f[g]$, we find
\begin{eqnarray*}
T(f,g) & = & \nabla_f g -\nabla_g f  -[f,g]_\bpr = gf - fg = -[f,g]\\
R(f,g)h & = & \nabla_f \nabla_g h-\nabla_g \nabla_f h -\nabla_{[f,g]_\bpr} h = 0.
\end{eqnarray*}
Note that if $\g$ is commutative, then $[f,g]=0$ and the connection is both flat and torsion free. In this case $f[g]$ is a pre-Lie product generating the Jacobi bracket: $f[g]-g[f]=[f,g]_\bpr$, but in general $f[g]-g[f]=[f,g]_\bpr-[f,g]$.

The product $f\bpr g$ is associative, and thus $U(\g)^\M$ with the binary operations $f,g\mapsto f[g]$ and $f,g\mapsto f\bpr g$ forms a unital dipterous algebra~\cite{loday08cha}, however, it has more structure than this. Following~\cite{munthe-kaas08oth} we define:

\begin{definition} Let $\A = \one\oplus \overline{\A}$ be a unital associative algebra with product $f,g\mapsto fg$, and also equipped with a non-associative composition $f,g\mapsto f[g]\colon \A\xpr \A\rightarrow\A$. Let $D(\A)$ denote all $f\in\A$ such that $f[\cdot]$ is a derivation: \[D(\A) = \stset{f\in \A}{f[gh] = (f[g])h + g(f[h])}.\]
We assume that $D(\A)$ generates $\overline{\A}$. We call $\A$ a D-algebra if for any derivation $f\in \D(\A)$ and any $g,h\in \A$ we have
\begin{align}
g[f] & \in D(\A)\\
\one[g] &= g\\
f[g[h]] &= (fg)[h] + (f[g])[h].
\end{align}
\end{definition}

\begin{definition}A D-algebra homomorphism is a map $\F\colon \A\rightarrow \A'$ between D-algebras such that $\F(D(\A))\subset D(\A')$ and for all $g,h\in \A$ we have
\begin{align}\F(\one) &= \one\\
\F(gh) &= \F(g)\F(h)\\
\F(g([h]) &= \F(g)[\F(h)].
\end{align}
\end{definition}

\subsubsection{The free D-algebra and elementary differentials}
The following definitions are detailed in~\cite{munthe-kaas08oth}.  Let $\OT$ denote the alphabet of all ordered (planar) rooted trees:
\[\OT = \{\ab,\aabb,\aababb,\aaabbb, \aabababb, \aabaabbb, \aaabbabb, \aaababbb,\ldots \}.\]
More generally, we consider decorated ordered rooted trees, where $\C$ is a (finite or infinite) set of colors. Decorated  trees are trees with a color from $\C$ assigned to  each node. As above, we let $\OT^*$ denote words of trees (forests), let $\one$ be the empty word and let $\omega_1,\omega_2\mapsto \omega_1\omega_2$ denote concatenation for $\omega_1,\omega_2\in \OT^*$.
Identifying $\C\subset\OT$ with 1-node trees, we can recursively build all words in $\OT^*$ from $\C$ by concatenation and adding roots. For $c\in\C$ and $\omega\in \OT^*$, define $B^+_{c}(\omega) \in \OT$ as the tree with branches $\omega$ and root $c$. Often we will be interested in the case where $\C = \{\ab\}$, just one color.

As above, let $\RR\langle\OT\rangle$ denote real polynomials (finite $\RR$-linear combinations of words) and $\RR\langle\langle\OT\rangle\rangle$ the dual space of infinite series, such as \[\alpha = \alpha(\one)\one + \alpha(\ab)\ab +\alpha(\aabb)\aabb+ \alpha(\ab\ab)\ab\ab + \alpha(\aaabbb)\aaabbb+\alpha(\aababb)\aababb+\alpha(\ab\aabb)\ab\aabb+ \alpha(\aabb\ab)\aabb\ab +\alpha(\ab\ab\ab)\ab\ab\ab+\cdots .\]
\noindent On $\RR\langle\OT\rangle$ we define \emph{left grafting} $(\cdot)[\cdot] \colon\RR\langle\OT\rangle\xpr \RR\langle\OT\rangle\rightarrow\RR\langle\OT\rangle$ by extending the following definition for trees by linearity. For all $c\in \C$, all  $\tau\in \OT$ and all $\omega,\omega'\in \OT^*$ we define:
\begin{align*}
   \omega[c] &=B^+_{c}(\omega)\\
   \one[\omega] & = \omega\\
   \tau[\omega\omega'] & = \tau[\omega]\omega' + \omega(\tau[\omega'])\\
   \tau[\omega[\omega']] & = (\tau\omega)[\omega'] + (\tau[\omega])[\omega']. 
 \end{align*}
 Compare this with Lemma~\ref{lem:dalgebra}. The left grafting $\tau[\omega]$ is obtained by attaching $\tau$ in all possible
 ways from the left to the vertices of $\omega$, and $(\tau\tau')[\omega]$ is obtained by attaching from the left first $\tau'$ and then $\tau$ on all nodes of $\omega$:  \begin{eqnarray*}
 \aABb\left[\ab\aabABb\aABb\right]&=&\aaABbb\aabABb\aABb+
 \ab\aaABbabABb\aABb+\ab\aaaABbbABb\aABb+\ab\aabAaABbBb\aABb+
 \ab\aabABb\aaABbABb+\ab\aabABb\aAaABbBb\\
 \ab\AB\left[\ab\aABb\right]&=&\aabABb\aABb+
 \aABb\aabABb+\aABb\aAabBb+\aabb\aABABb+\ab\aabABABb+
 \ab\aABAabBb+\aabb\aAABBb+\ab\aabAABBb+\ab\aAabABBb .
\end{eqnarray*}

\noindent We henceforth let $|\omega|$ denote the grading counting the total number of nodes, i.e.\ $|c|=1$ for all $c\in \C$, $|\omega\omega'| = |\omega|+|\omega'|$ and $|\omega[\omega']| = |\omega|+|\omega'|$.

\begin{proposition}Let  $\OT$ be planar trees decorated with colors $\C$. Consider $\N = \RR\langle\OT\rangle$ with concatenation $\omega,\omega'\mapsto \omega\omega'$, left grafting $\omega,\omega'\mapsto \omega[\omega']$ and unit $\one$ as defined above. 
$\N$ is a free D-algebra over $\C$, such that for any $D$-algebra $\A$ and 
any map $\nu\colon\C\rightarrow D(\A)$ there exists a unique D-algebra homomorphism map $\F_\nu\colon \N\rightarrow \A$ such that $\F_\nu(c)  = \nu(c)$ for all $c\in \C$.
\begin{diagram}[labelstyle=\scriptstyle]
\C &\rInto& \N \\
\dTo^{\nu} && \dTo_{\exists\,! \,\, \F_{\nu}} \\
D(\A) &\rInto& \A
\end{diagram}

\end{proposition}

\begin{definition}
We define the \emph{ordered Grossman--Larson}\footnote{The GL product is usually defined in a similar way over non-planar trees.} product on $\N$  for all $\omega,\omega'\in \OT^*$ as \[\omega\bpr \omega' = B^-(\omega[B^+(\omega')]) .\] 
I.e.\ we add a root to $\omega'$, graft on $\omega$ and finally remove the root again.
\end{definition}

\begin{proposition}
The GL-product is associative and, for all $n,n',n''\in \N$, satisfies
\begin{align}
n[n'[n'']] &= (n\bpr n')[n'']\\
\F_\nu(n\bpr n') &= \F_\nu(n)\bpr \F_\nu(n') .
\end{align}
\end{proposition}

\begin{remark}
The classical setting of Cayley, Merson and Butcher is the case where $\M=\RR^n$ and $\g=\{\partial/\partial x_i\}\subset\XM$ is the standard commutative coordinate frame. The construction of Section~\ref{sec:noncomm} produces $U(\g)^\M$ as a D-algebra where the concatenation is commutative.  The connection is now flat and torsionless, and $f[g]$ becomes a pre-Lie product. The images of the trees $\F(\tau)$, for $\tau\in\OT$, are called the \emph{elementary differentials} in Butcher's theory (see \cite{butcher63cft}). These are explicitly given in~(\ref{eg:elmdiffB}). The images of the forests $\F(\omega)$, for $\omega\in\OT^*$, are called \emph{elementary differential operators} in Merson's theory (see \cite{merson57aom}). 
\end{remark}

\subsubsection{A generalized Connes--Kreimer Hopf algebra of planar trees}
We recall from~\cite{munthe-kaas08oth} the definition of the Hopf algebra $\Hn$. On the vector space $\RR\langle\OT\rangle$ we define the shuffle product $\sqcup$, and we define the coproduct $\Delta_N$ as the dual of the ordered GL product, such that
\begin{equation}\label{eq:HnConv}
(\alpha\bpr \beta)(\omega) = \sum_{(\omega)_{\Delta_N}}\alpha(\omega_{(1)})\beta(\omega_{(2)})\quad \mbox{for all $\alpha,\beta\in 
\RR\langle\langle\OT\rangle\rangle\}$.}\end{equation}
The motivation for this construction is the representation of $U(\XM)$ in terms of a frame $\g\subset\XM$ as $U(\g)^\M$. The shuffle product is the correct product to characterize which series in $\RR\langle\langle\OT\rangle\rangle$ represent vector fields on $\M$ and which represent diffeomorphisms. The composition in $U(\XM)$ appears as the product $\bpr$ on $U(\g)^\M$, thus with the coproduct $\Delta_N$ the convolution on  $\RR\langle\langle\OT\rangle\rangle$ represents composition in $U(\XM)$. 

It remains to give a precise characterization of $\Delta_N$ and the antipode in $\Hn$. As in the Connes--Kreimer case, both $\Delta_N$ and the antipode can be defined directly in terms of admissible cuts or in a recursive fashion. Recursively $\Delta_N$ is given as
 \begin{equation}
   \label{eq:delta-r}
   \begin{split}
     \Delta_N(\one)
     & = \one\tpr\one, \\
     \Delta_N(\omega\tau)
     & = \omega\tau\tpr\one+\Delta_N(\omega)
     \sqcup\cdot(I\tpr B^+_{c})\Delta_N(\omega_1),
   \end{split}
 \end{equation}
 where $\tau = B^+_{c}(\omega_1)\in\OT$, where  $\omega,\omega_1\in\OT^*$ and where  $\sqcup\cdot$ denotes  shuffle on the left and concatenation on the right:  $(\omega_1\tpr\tau_1)\sqcup\cdot(\omega_2\tpr\tau_2) = (\omega_1\sqcup\omega_2)\tpr(\tau_1\tau_2)$. The direct formula is
\begin{equation}
   \label{eq:delta-nr}
   \Delta_N(\omega)= 
   \sum_{\ell\in\mathrm{FALC}(\omega)} P^\ell(\omega)\tpr R^\ell(\omega),
 \end{equation}
where FALC denotes \emph{Full Admissible Left Cuts}, $P^\ell(\omega)$ is the shuffle of all the cut off parts, and $R^\ell(\omega)$ is the remaining part containing the root (see \cite{munthe-kaas08oth}). Calculations of the coproduct for forests up to order 4 can be found in Table \ref{tab:coproduct}.

\begin{theorem} 
 \label{thm:hopf} 
 Let $\Hn$ be the vector space $\N = \mathbb{R}\langle \OT \rangle$ with the operations
 \begin{alignat*}{2}
   \mbox{product} & \colon \mu_N(a\tpr b) = a\sqcup b,
   \\
   \mbox{coproduct} & \colon \Delta_N,
   \\
   \mbox{unit}       & \colon  u_N(1) = \one,\\
   \mbox{counit} & \colon e_N(\omega) = \left\{\begin{array}{cl}
       1, & \mbox{if $\omega=\one$},\\
       0, & \mbox{else}.
     \end{array}\right. 
 \end{alignat*}
 Then $\Hn$ is a Hopf algebra with an antipode $S_N$ given by the recursion
 \begin{equation}
   \begin{split}
     \label{eq:antipodrecur}
     \S_N(\one) & = \one,\\
     \S_N(\omega\tau) & = -\mu_N\left((\S_N\tpr
       I)\left(\Delta_N(\omega) \sqcup\cdot
     (I\tpr B^+_i)\Delta_N(\omega_1)\right)\right),
   \end{split}
 \end{equation}
 where $\tau = B_i^+(\omega_1)\in\OT$ and $\omega,\omega_1\in\OT^*$.
\end{theorem}

 \subsubsection{Lie--Butcher series and flows on manifolds}
The set of maps $U(\g)^{\M}$ from $\M$ to $U(\g)$ is a D-algebra where the derivations are the vector fields $\g^\M$. Thus, given a set of colors $\C$ and a map $\nu\colon \C\rightarrow \g^\M$ there exists a unique map $\F_\nu\colon\N\rightarrow U(\g)^\M$ such that for all $c\in\C$ and all $g,h\in\N$ we have
 \begin{align}
 \F_\nu(c) & = \nu(c) \\
 \F_\nu(\one) &= \one\\
 \F_\nu(gh) &= \F_\nu(g)\F_\nu(h)\\
   \F_\nu(g[h]) &= \F_\nu(g)[\F_\nu(h)] \\
 \F_\nu(g\bpr h) &= \F_\nu(g)\bpr \F_\nu(h) .\\
 \end{align}

\begin{definition}~\label{eq:lbseries}
For an infinite series $\alpha\in \N^*=\RR\langle\langle\OT\rangle\rangle$
a Lie--Butcher series is a formal series in $U(\g)^\M$ defined as\[\Bs_{t}(\alpha) = \sum_{\omega\in\OT^*}t^{|\omega|}\alpha(\omega) \F_\nu(\omega). \]
\end{definition}

\noindent Note that $\N$ can be turned into a Hopf algebra two different ways: either as $\Hsh$ with product $\sqcup$ and deconcatenation coproduct $\cpd$, or as $\Hn$ with the same product $\sqcup$, but where the coproduct $\Delta_N$ is the dual of the ordered GL product. This gives rise to two different convolutions on $\N^*$,
the frozen composition $\alpha,\beta\mapsto \alpha\beta$ in Example~\ref{ex:shufalg2}, and the non-frozen composition
$\alpha,\beta\mapsto \alpha\bpr \beta$ as in~(\ref{eq:HnConv}). Since the product is the same, we have that the characters and the infinitesimal characters are the same as vector spaces
\begin{eqnarray*}\g(\Hsh)=\g(\Hn)&  =& \stset{\alpha\in \N}{\alpha(\one)=0, \alpha(\omega\sqcup\omega') = 0 \mbox{ for all 
$\omega,\omega'\in \OT^*\backslash \one$}}\\
G(\Hsh)=G(\Hn)&  =& \stset{\alpha\in \N}{\alpha(\one)=1, \alpha(\omega\sqcup\omega') = \alpha(\omega)\alpha(\omega') \mbox{ for all 
$\omega,\omega'\in \OT^*$}}.
\end{eqnarray*}
Hovever, the exponential, logarithm, Dynkin and Eulerian idempotents, as well as the antipode depend on whether they are based on $\Hsh$ or $\Hn$. Which to use in practice depends on which operation we want to express on the manifold. Recall that \emph{frozen elements} of $U(\g)^\M$ are
constant functions $g\colon \M\rightarrow U(\g)$. If $g$ is frozen then $f[g]=0$ for all $f$, and hence $f\bpr g = fg$. The subalgebra of frozen vector fields therefore reduces to $\Hsh$.

We summarize the basic properties of LB-series: $\Bs_t$ sends infinitesimal characters to (formal) vector fields on $\M$ and characters to pullback series representing formal diffeomorphisms on $\M$. LB-series preserve both frozen and non-frozen composition and sends left grafting to the connection on $U(\g)^\M$.
\begin{align*}
\Bs_t(\alpha\beta) &= \Bs_t(\alpha)\Bs_t(\beta)\\
\Bs_t(\alpha\bpr\beta) &= \Bs_t(\alpha)\bpr\Bs_t(\beta)\\
\Bs_t(\alpha[\beta])&= \Bs_t(\alpha)[\Bs_t(\beta)] .
\end{align*}
Note  that if $\alpha\in G(\Hn)$, then $\alpha[\beta]$ represents algebraically the pullback (parallel transport) of $\beta$ along the flow of $\alpha$. On 
the manifold 
\[\Bs_h(\alpha[\beta])(y_0)= \Bs_h(\alpha)[\Bs_h(\beta)](y_0) = \Bs_h(\beta)(\Phi(y_0)), \]
where $\Phi$ is the diffeomorphism represented by $\alpha\in G(\Hn)$ at $t=h$.
Since the connection is flat, the pullback depends only on the endpoint $\Phi(y_0)$ and not on the actual path.

There are (at least) three ways to represent a flow $y_0\mapsto y_t=\Phi_t(y_0)$ on $\M$,  using LB-series:
\begin{enumerate}
\item In terms of pullback series. Find $\alpha\in G(\Hn)$ such that 
  \begin{equation}
\psi(y(t)) = \Bs_t(\alpha)[\psi](y_0) \quad\mbox{for any $\psi\in U(\g)^\M$.}
\end{equation}
This representation is used in the analysis of Crouch--Grossman methods by Owren and Marthinsen~\cite{owren99rkm}.
In the classical setting, this is called a $S$-series~\cite{murua1999fsa}.
\item In terms of an autonomous differential equation. Find $\beta\in \g(\Hn)$ such that $y(t)$ solves
\begin{equation}
y'(t) = \Bs_h(\beta)(y(t))\dpr y(t).
\end{equation}
In the classical setting, this is called backward error analysis. In the Lie group setting, this formulation has, however, never been investigated in detail (but it should!).
\item In terms of a non-autonomous equation of \emph{Lie type} (time dependent frozen vector field). Find $\gamma\in \g(\Hsh)$ such that $y(t)$ solves
\begin{equation}\label{lietype}
y'(t) = \frac{\partial}{\partial t} \Bs_t(\gamma)(y_0)\dpr y(t).
\end{equation}
This representation is used in~\cite{munthe-kaas95lbt,munthe-kaas98rkm}. In the classical setting this is (almost) the standard definition of $B$-series. The connection with the classical B-series is discussed below. 
\end{enumerate}
The algebraic relationship between $\alpha$, $\beta$ and $\gamma$ is given as follows:
\begin{align*}
\beta &= \alpha\opr e &\mbox{$e$ is Euler idempotent in $\Hn$.}\\
\alpha &= \exp^\bpr(\beta)&\mbox{Exponential wrt.\ GL-product}\\
\gamma &= \alpha\opr Y^{-1}\opr D &\mbox{Dynkin idempotent in $\Hsh(\OT)$.}\\
\alpha &= Q(\gamma)&\mbox{$Q$-operator~(\ref{eq:Q}) in $\Hsh(\OT)$.}
\end{align*}

\begin{example}Two examples are of particular interest; the exact solution and exponential Euler method.
In both cases we consider $y'(t) = f(y)\dpr y$, where $\C = \{\ab\}$ and $\nu(\ab) = f$.

The exponential Euler method is particularly simple.  Since each step of the method follows the flow the frozen vector field $f(y_n)\in \g$, the
Type 3 LB-series for  Exponential Euler must be given by
\[ \gamma_{\text{Euler}} = \ab\]
just as in the classical setting\footnote{The classical presentation is $\gamma = \one+\ab$, when the B-series is given in the form~(\ref{eq:bserclassic}).}.

Type 3 LB-series for the exact solution can be derived in various ways. Theorem 2.2 in~\cite{munthe-kaas98rkm} derives the exact solution as the solution of
\[y' = f_t\dpr y, \quad y(0)= y_0,  \]
where $f_t=f(y(t))\in \g$ is the pullback of $f$ along the time dependent flow of $f_t$. Letting \\
$f_t = \frac{\partial}{\partial t}\Bs_t(\gamma)$ we obtain
\[
Y\opr\gamma = Q(\gamma)[\ab]\Rightarrow \gamma = Y^{-1}\opr B^+(Q(\gamma)) .
\]
Note that this is reminiscent of a so-called combinatorial Dyson--Schwinger equation~\cite{foissy2008fdb}. Solving by iteration yields
\begin{eqnarray*}
\gamma_{\text{Exact}} & = &
\ab + \frac{1}{2!}\aabb + \frac{1}{3!}(\aababb+\aaabbb) + \frac{1}{4!}(\aabababb+\aaabbabb+2\aabaabbb
+\aaababbb+\aaaabbbb)+ \frac{1}{5!}(\aababababb+\aaabbababb+ 2\aabaabbabb\\
& &
+3\aababaabbb+\aaababbabb+ \aaaabbbabb+3\aaabbaabbb+3\aabaababbb+3\aabaaabbbb+\aaabababbb+\aaaabbabbb+2\aaabaabbbb+\aaaababbbb+\aaaaabbbbb
)+\\
& &
\frac{1}{6!}(\aabababababb+\cdots)+\cdots
\end{eqnarray*}
Remarkably, the LB-series of the exact solution is just a combination of trees, and not commutators of trees. Thus in Type 3 LB-series developments of numerical integrators,  commutators of trees must be zero up to to the order of the method.
\end{example}

Composition and inverse is simplest for pullback series, Type 1. For series of Type 3, we map to Type 1, compose (or invert) and map back again. If $\gamma,\tilde{\gamma}$ are series of Type 3, then the basic operations are done as:
\begin{align}
\text{Composition} &\colon\quad \gamma,\tilde{\gamma}\mapsto(Q(\gamma)\bpr Q(\tilde{\gamma}))\opr Y^{-1}\opr D \\
\text{Inverse} & \colon\quad \gamma^{-1} = Q(\gamma)\opr S\opr Y^{-1}\opr D\label{eq:3inv} \\
\text{Backward error} &\colon\quad \Log_3(\gamma):= Q(\gamma)\opr e .\label{eq:3log}
\end{align}

\subsubsection{Relations to classical B-series}The relation between classical B-series and LB-series is detailed in~\cite{munthe-kaas08oth}. 
Classical $B$-series are expressed in terms of linear combinations of non-planar trees $T$, resulting in the Connes--Kreimer Hopf algebra $\Hc$ built from non-planar trees \cite{brouder2000rkm}.
In the classical setting the connection is torsion-free, and concatenation is commutative. Therefore $\g(\Hc) = \text{span}(T)$. That is, $\g(\Hc)$ is just linear combinations of trees. This fact is the reason why many discussions in the classical setting can avoid series involving forests of trees (words in $T^*$). Also the difference between series of Type 1 and Type 3 is in not emphasized in many papers. Since the coefficients $\kappa$ of the $Q$-polynomials add up to one under symmetrization, we find in the classical setting that \[Q(\alpha)(\omega) = \alpha(\tau_1)\alpha(\tau_2)\cdots\alpha(\tau_k)\omega ,\]
for $\omega=B^+(\tau_1\tau_2\cdots \tau_k)$, so formulas involving pullbacks are often expressed directly from B-series (Type 3) using the $Q$-polynomials in this form. Our claim that classical B-series fits best into series of Type 3 is based on the trivial observation that the curve $y_t = \Bs_t(\alpha)(y)$ in~(\ref{lietype}) solves a differential equation with a time dependent frozen vector field given as
\[y(t) = \frac{\partial}{\partial t} \sum_{\tau\in T} \frac{t^{|\tau|}}{\sigma(\tau)}\F(\tau) .\]
\noindent One can ask why the symmetrization $\sigma(\tau)$ is natural to include in the classical setting, but not in the LB-series setting. To explain the relationship between the two theories we define a symmetrization operator:

\begin{definition}
 \label{defn:symm}
 The symmetrization operator $\signature{\Omega}{\N}{\N}$ is defined for $\omega\in \OT^*$ and $\tau\in \OT$ as
 \begin{align*}
   \Omega(\one) & = \one,\\
   \Omega(\omega\tau) & = \Omega(\omega)\sqcup\Omega(\tau),\\
   \Omega(B^+_i(\omega)) & = B^+_i(\Omega(\omega)).
 \end{align*}
\end{definition}
\noindent The shuffle product permutes the trees in a forest in all possible ways, and the symmetrization of a tree is a recursive splitting in sums over all permutations of the branches.  The symmetrization defines an equivalence relation on $\OT^*$, that is
\begin{equation*}
 \label{eq:eqforests}
 \Omega(\omega_1)=\Omega(\omega_2) 
 \quad\Longleftrightarrow\quad \omega_1\sim\omega_2.
\end{equation*}
\noindent Let $\iota:\Hc\rightarrow\Hn$ be an inclusion where a tree is identified with one of its equivalent planar trees. In~\cite{munthe-kaas08oth} we show that $\tilde{\Omega} = \Omega\opr \iota\colon \Hc\rightarrow\Hn$ is a Hopf algebra isomorphism onto its image, i.e.\ $\Hc$ is a proper subalgebra of $\Hn$. The adjoint map $\tilde{\Omega}^*\colon \Hn^*\rightarrow\Hc^*$ is given as \[\tilde{\Omega}^*(\alpha)(\omega) = \sigma(\omega)\sum_{\omega'\sim \omega}\alpha(\omega').\] The tree symmetrization $\sigma(\omega)$ enters exactly such that the LB-series as given in~(\ref{eq:lbseries}) maps to the classical B-series in~(\ref{eq:bserclassic}).

\subsection{Substitution law for LB-series}
The so-called \emph{substitution law} for B-series~\cite{chartier2005asl} can without much difficulty be generalized to LB series. Consider $\N$ as a D-algebra where the derivations are the Lie polynomials $D(\N) = \g(\Hn)\cap \N$. By the universality property of $\N$, we know that for any map  $a\colon \C\rightarrow D(\N)$ there exists a unique D-algebra homomorphism
$\F_a:\N\rightarrow \N$ such that $\F_a(c)=a(c)$ for all $a\in \C$. This is called the substitution law.

\begin{multicols}{2}
\begin{definition}For any map $a\colon\C\rightarrow D(\N)$ there exists a unique D-algebra homomorphism 
$a\star:\N\rightarrow \N$ such that $a(c) = a\star c$ for all $c\in \C$. The map $a\star$ is called $a$-substitution\footnote{In most applications we want to substitute infinite series and extend $a\star$ to a homomorphism $a\star\colon\N^*\rightarrow \N^*$. 
The extension to infinite substitution is straightforward because of the grading, we omit details.}.
\newline
\begin{diagram}[labelstyle=\scriptstyle]
\C &\rInto& \N\\
\dTo^a && \dTo_{a\star} \\
D(\N) &\rInto& \N.
\end{diagram}
\end{definition}
\end{multicols}

The properties of this substitution law, together with applications of it, will be studied in a forthcoming paper (\cite{lundervold2010bea}). We just mention that many of the useful properties of the substitution law follow immediately from the fact that $a\star:\N\rightarrow \N$ is a homomorphism. For example, for all $n,n'\in\N$ we have:
\begin{align*}
a\star\one &= \one\\
a\star (nn') &= (a\star n)(a\star n')\\
a\star (n[n']) &= (a\star n)[a\star n']\\
a\star (n\bpr n') &= (a\star n)\bpr (a\star n')\\
\end{align*}

\section{Final remarks and outlook}
Inspired by problems in numerical analysis we have discussed various algebraic structures arising in the study of formal diffeomorphisms on manifolds. We have seen that the Connes--Kreimer Hopf algebra naturally extends from commutative frames on $\RR^n$ to non-commutative frames on general manifolds. In particular we have presented the Dynkin and Euler operators and non-commutative Fa\`a di Bruno type bialgebras in this generalized setting.

The formalism in this paper has many applications in numerical analysis, and analysis of Lie group integrators in particular. However, the  underlying structures are general constructions with possible applications in other fields, such as geometric control theory and sub-Riemannian geometry. Connections to stochastic differential equations on manifolds is an other topic which is worth investigating further.

\section*{Acknowledgements} The authors would like to thank Alessandra Frabetti, Dominique Manchon, Gilles Vilmart and Will Wright for interesting discussions on topics of this paper. In particular we would like to thank Kurusch Ebrahimi-Fard for his support and useful remarks in the writing process. His enthusiasm and inclusive spirit have been of crucial importance for the completion of this paper.

\newpage

\begin{table}[!ht]
 \centering
 \begin{equation*}
   \begin{array}{c@{\,\,}|@{\quad}l}
     \hline \\[-2mm]
     \omega & \Delta_N(\omega)  \\[1mm] 
     \hline \\[-2mm]
     \one & \one\tpr\one \\[1mm] 
     \ab & \ab\tpr\one+\one\tpr\ab \\[2mm]     
     \aabb & \aabb\tpr\one+\ab\tpr\ab+\one\tpr\aabb  \\[2mm]
     \ab\ab & \ab\ab\tpr\one+\ab\tpr\ab+\one\tpr\ab\ab \\[2mm]
     \aaabbb &\aaabbb\tpr\one+\ab\tpr\aabb+\aabb\tpr\ab+
     \one\tpr\aaabbb \\[2.5mm]
     \aababb & \aababb\tpr\one+\ab\ab\tpr\ab+\ab\tpr\aabb+
     \one\tpr\aababb \\[2.5mm]
     \ab\aabb & \ab\aabb\tpr\one+2\ab\ab\tpr\ab
     +\ab\tpr\aabb+\ab\tpr\ab\ab+
     \one\tpr\ab\aabb \\[2.5mm]
     \aabb\ab & \aabb\ab\tpr\one+\aabb\tpr\ab+
     \ab\tpr\ab\ab+\one\tpr\aabb\ab \\[2mm]
     \ab\ab\ab & \ab\ab\ab\tpr\one+\ab\ab\tpr
     \ab+\ab\tpr\ab\ab+\one\tpr\ab\ab\ab \\[2mm]
     \aaaabbbb & \aaaabbbb\tpr\one+\aaabbb\tpr\ab+\aabb\tpr\aabb+
     \ab\tpr\aaabbb+\one\tpr\aaaabbbb \\[2.5mm]
     \aaababbb & \aaababbb\tpr\one+\aababb\tpr\ab+\ab\ab\tpr\aabb+
     \ab\tpr\aaabbb+\one\tpr\aaababbb \\[2.5mm]
     \aabaabbb & \aabaabbb\tpr\one+\ab\aabb\tpr\ab+2\ab\ab\tpr
     \aabb+\ab\tpr\aaabbb+\ab\tpr\aababb+\one\tpr\aabaabbb \\[2.5mm]
     \aaabbabb & \aaabbabb\tpr\one+\aabb\ab\tpr\ab+\aabb\tpr
     \aabb+\ab\tpr\aababb+\one\tpr\aaabbabb \\[2.5mm]
     \aabababb & \aabababb\tpr\one+\ab\ab\ab\tpr\ab+\ab\ab\tpr
     \aabb+\ab\tpr\aababb+\one\tpr\aabababb \\[2.5mm]
     \ab\aaabbb & \ab\aaabbb\tpr\one+\ab\aabb\tpr\ab+\aabb\ab\tpr\ab+\aabb\tpr
     \ab\ab+2\ab\ab\tpr\aabb+\ab\tpr\ab\aabb+
     \ab\tpr\aaabbb+\one\tpr\ab\aaabbb \\[2.5mm]
     \aaabbb\ab & \aaabbb\ab\tpr\one+\aaabbb\tpr\ab+\aabb\tpr
     \ab\ab+\ab\tpr\aabb\ab+\one\tpr\aaabbb\ab \\[2.5mm]
     \ab\aababb & \ab\aababb\tpr\one+3\ab\ab\ab\tpr\ab+
     \ab\ab\tpr\ab\ab+2\ab\ab\tpr\aabb+
     \ab\tpr\ab\aabb+\ab\tpr\aababb+
     \one\tpr\ab\aababb \\[2.5mm]
     \aababb\ab &  \aababb\ab\tpr\one+\aababb\tpr\ab+\ab\ab\tpr
     \ab\ab+\ab\tpr\aabb\ab+\one\tpr\aababb\ab \\[2.5mm]
     \aabb\aabb & \aabb\aabb\tpr\one+\aabb\ab\tpr\ab+\ab\aabb\tpr\ab
     +\aabb\tpr\aabb+2\ab\ab\tpr\ab\ab+
     \ab\tpr\ab\aabb+\ab\tpr\aabb\ab+
     \one\tpr\aabb\aabb \\[2.5mm]
     \ab\ab\aabb & \ab\ab\aabb\tpr\one+3\ab\ab\ab\tpr\ab+
     2\ab\ab\tpr\ab\ab+\ab\ab\tpr\aabb+
     \ab\tpr\ab\ab\ab+\ab\tpr\ab\aabb+
     \one\tpr\ab\ab\aabb \\[2.5mm]
     \ab\aabb\ab & \ab\aabb\ab\tpr\one+\ab\aabb\tpr\ab+
     2\ab\ab\tpr\ab\ab+\ab\tpr\ab\ab\ab+
     \ab\tpr\aabb\ab+\one\tpr\ab\aabb\ab \\[2.5mm]
     \aabb\ab\ab & \aabb\ab\ab\tpr\one+\aabb\ab\tpr\ab+
     \aabb\tpr\ab\ab+\ab\tpr\ab\ab\ab+
     \one\tpr\aabb\ab\ab \\[2mm]
     \ab\ab\ab\ab & \ab\ab\ab\ab\tpr\one+\ab\ab\ab\tpr\ab+
     \ab\ab\tpr\ab\ab+\ab\tpr\ab\ab\ab+
     \one\tpr\ab\ab\ab\ab \\[2mm] \hline
   \end{array}
 \end{equation*}
 \caption{Examples of the coproduct $\Delta_N$, defined
   in \eqref{eq:delta-r}. \label{tab:coproduct}} 
\end{table}

\bibliography{./bib/ref_alex,./bib/hopf,./bib/glm,./bib/geom_int}

\begin{thebibliography}{10}

\bibitem{abe80ha}
E.~Abe.
\newblock {\em {Hopf Algebras}}.
\newblock Cambridge University Press, 1980.

\bibitem{abraham88mta}
R.~Abraham, J.~E. Marsden, and T.~Ratiu.
\newblock {\em Manifolds, Tensor Analysis, and Applications}.
\newblock {AMS} 75. Springer-Verlag, {S}econd edition, 1988.

\bibitem{berland05aso}
H.~Berland and B.~Owren.
\newblock {Algebraic structures on ordered rooted trees and their significance
  to Lie group integrators}.
\newblock {\em Group theory and numerical analysis}, 39:49--63, 2005.

\bibitem{brouder2000rkm}
C.~Brouder.
\newblock {Runge-Kutta methods and renormalization}.
\newblock {\em The European Physical Journal C-Particles and Fields},
  12(3):521--534, 2000.

\bibitem{brouder04tra}
C.~Brouder.
\newblock Trees, renormalization and differential equations.
\newblock {\em BIT}, 44(3):425--438, 2004.

\bibitem{brouder06nch}
C.~Brouder, A.~Frabetti, and C.~Krattenthaler.
\newblock {Non-commutative Hopf algebra of formal diffeomorphisms}.
\newblock {\em Advances in Mathematics}, 200(2):479--524, 2006.

\bibitem{burgunder2008eia}
E.~Burgunder.
\newblock {Eulerian idempotent and Kashiwara-Vergne conjecture}.
\newblock 58(4):1153--1184, 2008.

\bibitem{butcher63cft}
J.~C. Butcher.
\newblock Coefficients for the study of {R}unge-{K}utta integration processes.
\newblock {\em J. Austral. Math. Soc.}, 3:185--201, 1963.

\bibitem{butcher72aat}
J.~C. Butcher.
\newblock An algebraic theory of integration methods.
\newblock {\em Math. Comp.}, 26:79--106, 1972.

\bibitem{calaque08tha}
D.~Calaque, K.~Ebrahimi-Fard, and D.~Manchon.
\newblock {Two interacting Hopf algebras of trees}.
\newblock {\em To appear in Adv. Appl. Math}, 2009, math.CO/0806.2238v3.

\bibitem{cartier2006apo}
P.~Cartier.
\newblock {A primer of Hopf algebras}.
\newblock In {\em {Frontiers in number theory, physics, and geometry}},
  volume~II, pages 537--615. Springer, Berlin, 2007.

\bibitem{cayley1857taf}
A.~Cayley.
\newblock {On the theory of the analytical forms called trees}.
\newblock {\em Philos. Mag}, 13(19):4--9, 1857.

\bibitem{celledoni2003cfl}
E.~Celledoni, A.~Marthinsen, and B.~Owren.
\newblock {Commutator-free Lie group methods}.
\newblock {\em Future Generation Computer Systems}, 19(3):341--352, 2003.

\bibitem{chartier2005asl}
P.~Chartier, E.~Hairer, and G.~Vilmart.
\newblock {A substitution law for B-series vector fields}.
\newblock {\em INRIA report}, (5498), 2005.

\bibitem{chartier07nib}
P.~Chartier, E.~Hairer, and G.~Vilmart.
\newblock {Numerical integrators based on modified differential equations}.
\newblock {\em Mathematics of Computation}, 76(260):1941, 2007.

\bibitem{chartier08aat}
P.~Chartier and A.~Murua.
\newblock {An algebraic theory of order}.
\newblock {\em ESAIM: Mathematical Modelling and Numerical Analysis},
  43(4):607--630, 2009.

\bibitem{connes1998har}
A.~Connes and D.~Kreimer.
\newblock {Hopf algebras, renormalization and noncommutative geometry}.
\newblock {\em Communications in Mathematical Physics}, 199(1):203--242, 1998.

\bibitem{crouch93nio}
P.~E. Crouch and R.~Grossman.
\newblock Numerical integration of ordinary differential equations on
  manifolds.
\newblock {\em J. Nonlinear Sci.}, 3:1--33, 1993.

\bibitem{dur86mfi}
A.~D{\"u}r.
\newblock {\em M\"obius functions, incidence algebras and power series
  representations}, volume 1202 of {\em Lecture Notes in Mathematics}.
\newblock Springer-Verlag, Berlin, 1986.

\bibitem{ebrahimi-fard07alt}
K.~Ebrahimi-Fard, J.M. Gracia-Bond{\'\i}a, and F.~Patras.
\newblock {A Lie Theoretic Approach to Renormalization}.
\newblock {\em Communications in Mathematical Physics}, 276(2):519--549, 2007.

\bibitem{ebrahimi-fard2009ama}
K.~Ebrahimi-Fard and D.~Manchon.
\newblock {A Magnus-and Fer-type formula in dendriform algebras}.
\newblock {\em Foundations of Computational Mathematics}, 9:1--22, 2009,
  math.CO/07070607v3.

\bibitem{figueroa2005cha}
H.~Figueroa and J.M Gracia-Bondia.
\newblock {Combinatorial Hopf algebras in quantum field theory I}.
\newblock {\em Rev.Math.Phys.}, 17:881, 2005, hep-th/0408145v3.

\bibitem{figueroa2005fdb}
H.~Figueroa, J.M. Gracia-Bondia, and J.C. Varilly.
\newblock {Faa di Bruno Hopf algebras}.
\newblock {\em Preprint}, 2005, math.CO/0508337.

\bibitem{foissy2008fdb}
L.~Foissy.
\newblock {Fa{\`a} di Bruno subalgebras of the Hopf algebra of planar trees
  from combinatorial Dyson--Schwinger equations}.
\newblock {\em Advances in Mathematics}, 218(1):136--162, 2008, 0707.1204v2.

\bibitem{hairer06gni}
E.~Hairer, C.~Lubich, and G.~Wanner.
\newblock {\em Geometric {N}umerical {I}ntegration}.
\newblock Springer-Verlag, second edition, 2006.

\bibitem{hairer74otb}
E.~Hairer and G.~Wanner.
\newblock On the {B}utcher group and general multi-value methods.
\newblock {\em Computing (Arch. Elektron. Rechnen)}, 13(1):1--15, 1974.

\bibitem{iserles1999imm}
A.~Iserles, A.~Marthinsen, and S.P. N{\o}rsett.
\newblock {On the implementation of the method of Magnus series for linear
  differential equations}.
\newblock {\em BIT Numerical Mathematics}, 39(2):281--304, 1999.

\bibitem{iserles00lgm}
A.~Iserles, H.Z. Munthe-Kaas, S.P. N{\o}rsett, and A.~Zanna.
\newblock Lie-group methods.
\newblock {\em Acta Numerica}, 9:215--365, 2000.

\bibitem{iserles1999sld}
A.~Iserles and S.P. N{\o}rsett.
\newblock {On the solution of linear differential equations in Lie groups}.
\newblock {\em Philosophical Transactions of the Royal Society A: Mathematical,
  Physical and Engineering Sciences}, 357(1754):983--1019, 1999.

\bibitem{kassel95qg}
C.~Kassel.
\newblock {\em Quantum groups}.
\newblock Springer-Verlag, 1995.

\bibitem{lenczewski1998anl}
R.~Lenczewski.
\newblock {A noncommutative limit theorem for homogeneous correlations}.
\newblock {\em Studia Mathematica}, 129(3), 1998.

\bibitem{loday97ch}
J.L. Loday.
\newblock {\em Cyclic Homology}.
\newblock Springer-Verlag, second edition, 1997.

\bibitem{loday08cha}
J.L. Loday and M.O. Ronco.
\newblock {Combinatorial Hopf algebras}.
\newblock {\em Clay Mathematics Proceedings}, 12:347--384, 2010,
  math.CO/0508337.

\bibitem{lundervold2010bea}
A.~Lundervold and H.Z. Munthe-Kaas.
\newblock {Backward error analysis and the substitution law for Lie group
  integrators}.
\newblock {\em Preprint}, 2010.

\bibitem{manchon06haf}
D.~Manchon.
\newblock {Hopf algebras, from basics to applications to renormalization}.
\newblock {\em Preprint}, 2006, math.QA/0408405v2.

\bibitem{merson57aom}
R.~H. Merson.
\newblock An operational method for the study of integration processes.
\newblock In {\em Proc. Conf., Data Processing \& Automatic Computing
  Machines}, pages 110--1--11025, 1957.

\bibitem{monaco07fcc}
S.~Monaco, D.~Normand-Cyrot, and C.~Califano.
\newblock {From chronological calculus to exponential representations of
  continuous and discrete-time dynamics: a Lie-algebraic approach}.
\newblock {\em IEEE Transactions on Automatic Control}, 52(12):2227--2241,
  2007.

\bibitem{munthe-kaas95lbt}
H.~Munthe-Kaas.
\newblock {L}ie--{B}utcher theory for {R}unge--{K}utta methods.
\newblock {\em BIT}, 35(4):572--587, 1995.

\bibitem{munthe-kaas98rkm}
H.~Munthe-Kaas.
\newblock {R}unge--{K}utta methods on {L}ie groups.
\newblock {\em BIT}, 38(1):92--111, 1998.

\bibitem{munthe-kaas03oep}
H.~Munthe-Kaas and S.~Krogstad.
\newblock {On enumeration problems in Lie--Butcher theory}.
\newblock {\em Future Generation Computer Systems}, 19(7):1197--1205, 2003.

\bibitem{munthe-kaas99cia}
H.~Munthe-Kaas and B.~Owren.
\newblock Computations in a free {L}ie algebra.
\newblock {\em R. Soc. Lond. Philos. Trans. Ser. A Math. Phys. Eng. Sci.},
  357(1754):957--981, 1999.

\bibitem{munthe-kaas08oth}
H.Z. Munthe-Kaas and W.~Wright.
\newblock On the {H}opf algebraic structure of {L}ie group integrators.
\newblock {\em Found. Comput. Math}, 8(2):227 -- 257, 2008, math/0603023v1.

\bibitem{murua1999fsa}
A.~Murua.
\newblock {Formal series and numerical integrators, Part I: Systems of ODEs and
  symplectic integrators}.
\newblock {\em Applied numerical mathematics}, 29(2):221--251, 1999.

\bibitem{owren05ocf}
B.~Owren.
\newblock {Order conditions for commutator-free Lie group methods}.
\newblock {\em Journal of Physics A--Mathematical and General},
  39(19):5585--5600, 2006.

\bibitem{owren99rkm}
B.~Owren and A.~Marthinsen.
\newblock {R}unge--{K}utta methods adapted to manifolds and based on rigid
  frames.
\newblock {\em BIT}, 39(1):116--142, 1999.

\bibitem{patras02oda}
F.~Patras.
\newblock {On Dynkin and Klyachko idempotents in graded bialgebras}.
\newblock {\em Advances in Applied Mathematics}, 28(3/4):560--579, 2002.

\bibitem{reutenauer93fla}
C.~Reutenauer.
\newblock {\em Free Lie algebras}.
\newblock Oxford University Press, 1993.

\bibitem{sweedler69ha}
M.~E. Sweedler.
\newblock {\em Hopf algebras}.
\newblock Mathematics Lecture Note Series. W. A. Benjamin, Inc., New York,
  1969.

\end{thebibliography}

\end{document}